\documentclass[12pt]{article}

\usepackage[toc,page]{appendix}

\usepackage[utf8]{inputenc}
\usepackage{amssymb, amsfonts, amsmath, amsthm}
\usepackage{enumerate}
\usepackage{xcolor}

\usepackage{pgf,tikz}

\usepackage{fullpage}
\usepackage[colorlinks=true,linkcolor=blue,citecolor=blue]{hyperref}

\usepackage{libertine}
\usepackage{libertinust1math}
\usepackage[T1]{fontenc}

\usepackage{verbatim}
\usepackage{bm}

\usepackage[colorlinks=true,linkcolor=blue,citecolor=blue]{hyperref}
\usepackage[colorinlistoftodos,prependcaption,textsize=small]{todonotes}

\usepackage{caption}	
\usepackage{subcaption} 

\graphicspath{{figures/}} 

\newcommand{\cc}{\mathbb{C}}

\newcommand{\nn}{\mathbb{N}}

\newcommand{\rr}{\mathbb{R}}

\newcommand{\CC}{\mathcal{C}}

\newcommand{\GG}{\mathcal{G}}

\newcommand{\MM}{\mathcal{M}}

\newcommand{\blocks}[1]{| #1 |}

\newcommand{\partlat}{P}

\newcommand{\freec}[1]{r_{#1}}

\newcommand{\R}{\mathbb{R}}
\newcommand{\C}{\mathbb{C}}

\newcommand{\abs}[1]{\left\vert {#1} \right\vert}

\newcommand{\gdos}{H}

\newcommand{\rinf}{R^{inf}}

\newcommand{\mfp}{\mathfrak{p}}
\newcommand{\mfq}{\mathfrak{q}}

\newcommand{\pols}{\mathcal P}

\usepackage{stmaryrd}
\newcommand{\meas}[1]{\mu \left\llbracket #1 \right\rrbracket}

\newcommand{\coef}[2]{\mathsf{e}_{#1}\left( #2\right) }

\newcommand{\ffc}[2]{\kappa_{#1}^{(#2)}} 

\newtheorem{theorem}{Theorem}[section]
\newtheorem{lemma}[theorem]{Lemma}
\newtheorem{definition}[theorem]{Definition}

\newtheorem{corollary}[theorem]{Corollary}
\newtheorem{proposition}[theorem]{Proposition}

\newtheorem{notation}[theorem]{Notation}
\newtheorem{remark}[theorem]{Remark}

\newtheorem{example}[theorem]{Example}

\title{Finite-free Convolution: Infinitesimal Distributions}

\author{\normalsize Octavio Arizmendi  \\ \normalsize octavius@cimat.mx  \\ \normalsize CIMAT, Guanajuato  \and \normalsize Daniel Perales  \\\normalsize dperale2@nd.edu\\ \normalsize University of Notre Dame \and \normalsize Josue Vazquez-Becerra \\\normalsize jdvb@cimat.mx \\ \normalsize CIMAT, Guanajuato} 

\date{}

\begin{document}

\maketitle

\begin{abstract}

Finite-free additive and multiplicative convolutions are operations on the set of polynomials with real roots, introduced independently by Szeg\"{o} and Walsh in the 1920s. 
These operations have regained some interest, in the last decade, after being rediscovered  by Marcus, Spielman, and Srivastava as the expected characteristic polynomial of randomly rotated matrices. 
They converge, as the degree $d$ of the polynomials increases, to the additive and multiplicative convolution of measures from free probability of Voiculescu. 
In this paper, we investigate the fluctuations of order $1/d$ ---also known as infinitesimal distributions--- related to these two operations and their limiting behavior, providing a detailed description of their convergence. 
Our approach relies on understanding the infinitesimal moment-cumulant formulas and the corresponding functional relations.
We also establish several applications and examples, including instances related to the infinitesimal free convolution of Belinschi and Shlyakhtenko, as well as the computation of infinitesimal distributions after differentiation of polynomials.

\end{abstract}

\section{Introduction}
%
Let $\mathbb{R}_d[x]$ denote the set of all polynomials in $x$ of degree $d$ with real coefficients.
Given any polynomials $p_d,q_d\in \mathbb{R}_d[x] $, we write them as 
\begin{equation*}
p_d(x)=\sum_{k=0}^d x^{d-k}(-1)^k \binom{d}{k} \coef{k}{p_d} 
\qquad \text{and} \qquad  
q_d(x)=\sum_{k=0}^d x^{d-k}(-1)^k  \binom{d}{k} \coef{k}{q_d} 
\end{equation*}
%
for some unique coefficients 
$\coef{k}{p_d}, \coef{k}{q_d} \in \rr$. 
The \emph{finite-free additive convolution} of $p_d$ and $q_d$ is defined as 
the polynomial $p_d\boxplus_d q_d \in \mathbb{R}_d[x]$ given by 
\begin{equation*}
[p_d \boxplus_d q_d](x)
	=
		\sum_{k=0}^d x^{d-k}(-1)^k 
				\binom{d}{k} 
				\left[
				\sum_{i=0}^{k}\binom{k}{i} \coef{i}{p_d} \coef{k-i}{q_d}\right] .
\end{equation*}
Additionally, the \emph{finite-free multiplicative convolution} of $p_d$ and $q_d$ is
defined as the polynomial $p_d\boxtimes_d q_d \in \mathbb{R}_d[x]$ given by 
\begin{equation*}
[p_d \boxtimes_d q_d](x) =
			\sum_{k=0}^d x^{d-k}(-1)^k \binom{d}{k} \Big[ \coef{k}{p_d} \coef{k}{q_d} \Big].
\end{equation*}
A century ago, these two operations were studied independently by Szeg\"{o} \cite{szego1922bemerkungen} and Walsh \cite{walsh1922location}, and they have recently been rediscovered by Marcus, Spielman, and Srivastava \cite{marcus2016polynomial} as the expected characteristic polynomials of the sum and product of randomly rotated matrices. 
%
%
In this paper, we examine the $1/d$ behavior of these two polynomial convolutions. 

Let $\pols_d(\rr)$ denote the set of all monic and real-rooted polynomials in $\mathbb{R}_d[x]$. 
For any polynomial $p_d \in \pols_d(\rr)$, we write its $d$ roots counting multiplicities as 
\begin{equation*}
\lambda_d(p_d)  \leq 
\lambda_{d-1}(p_d) \leq \cdots \leq
\lambda_{1}(p_d) . 
\end{equation*}
And we identify each $p_d \in \pols_d(\rr)$ with the discrete probability distribution $\mu_{p_d}$ on $\rr$ that puts a mass of $1/d$ at each $\lambda_k(p_d)$, called the \emph{root distribution} of $p_d$, that is, 
\begin{equation*}
\mu_{p_d}=\frac{1}{d}\sum^d_{k=1} \delta_{\lambda_k(p_d)} . 
\end{equation*}
The set $\pols_d(\rr)$ is closed under finite-free additive convolution, i.e., $p_d,q_d\in \pols_d(\rr)$ implies $p_d\boxplus_d q_d\in \pols_d(\rr)$, so the finite-free additive convolution of real-rooted polynomials is also real-rooted, see \cite{walsh1922location}.
Real-rootedness is also preserved under multiplicative convolution if at least one of the polynomials has no negative roots. 
Concretely, letting $\pols_d(\rr_{\geq 0})$ denote the set of all polynomials in $\pols_d(\rr)$ with no negative roots, we have that $p_d\in \pols_d(\rr)$ and $q_d\in \pols_d(\rr_{\geq 0})$ implies $ p_d\boxtimes_d q_d\in \pols_d(\rr)$, see \cite{szego1922bemerkungen}
%
%
%
%

Let $\MM(\rr)$ denote the set of all probability distributions on $\rr$. 
Given a sequence of polynomials $\mfp = (p_d)_{d\geq1}$ with $p_d \in \pols_d(\rr)$ for each $d \geq 1$, 
we say that it has \emph{limiting (root) distribution} $\mu \in \MM(\rr)$ if the sequence $(\mu_{p_d})_{d\geq1}$ converges in moments to  $\mu$, namely, for every integer $n \geq 0$, we have that 
\begin{equation}
 \int_{\rr} t^n  \,  d \mu_{p_d}(t) \, = \frac{1}{d} \sum_{k=1}^{d} \lambda^n_k(p_d) 
\ \longrightarrow
\int_{\rr} t^n  \,  d \mu(t) = : m_n(\mu) 
\quad \text{as}  \quad 
d \to \infty. 
\end{equation}
Moreover, the \emph{$n$-th infinitesimal moment} of $\mfp = (p_d)_{d\geq1}$ is defined as 
\begin{equation}\label{eqn:infinitesimal_moment_intro}
m_n'(\mfp) = \lim_{d\to \infty} d \, \left( m_n(\mu_{p_d}) - m_n(\mu)  \right),
\end{equation}
if such limit exists.  
Note that in this case we have a truncated expansion in $1/d$ of the form 
\begin{equation}\label{eqn:infinitesimal_description}
m_n(\mu_{p_d}) = m_n(\mu) + \frac{1}{d} \, m_n'(\mfp) + o (d^{-1})
\end{equation}
where $o (d^{-k})$ denotes a corresponding correction term satisfying $\lim_{d \to \infty} d^k \cdot  o (d^{-k}) = 0$. 
When $\mfp = (p_d)_{d\geq1}$ has infinitesimal moments of all orders, we identify the sequence $(m_n'(\mfp))_{n\geq1}$ with the unique linear functional $\mu':\rr[x] \to \rr$ such that $\mu'(x^n) = m_n'(\mfp)$ for every $n \geq 0$. 
Note that we always have $\mu'(1)=0$. 
We refer to the pair $(\mu,\mu')$ as the \emph{infinitesimal distribution} of $\mfp = (p_d)_{d\geq1}$. 

The notion of infinitesimal distribution $(\mu,\mu')$ ---for a non-commutative variable--- appeared first in the work of Belinschi and Shlyakhtenko \cite{belinschi2012free}, as a generalization of the Type B distributions of Biane, Goodman, and Nica \cite{biane2003non} and only requiring $\mu$ and $\mu'$ to be linear functionals on $\C[x]$ with $\mu(1)=1$ and $\mu'(1)=0$. 
The idea behind it is to capture at the same time the zeroth and next-order behavior in the convergence of a sequence, examining not only the limiting object but also the way the sequence fluctuates around it.
Thus, related intrinsically to our work, we find the infinitesimal distributions $(\mu,\mu')$ coming from the eigenvalue distribution of large random matrices, see  \cite{johansson1998fluctuations,mingo2004annular,dumitriu2006global,mingo2019non}, where $\mu$ is a probability distribution on $\R$ and $\mu'$ is a signed measure on $\R$ with total mass $\mu'(\R)=0$.
In the present work, we take $\mu$ as a probability distribution on $\R$ and $\mu':\R[x] \to \R$  as a linear functional with $\mu'(1)=0$, and we will identify $\mu'$ with a signed measure when appropriate. 

\textbf{Main question.} 
If $\mfp=(p_d)_{d\geq1}$ and $\mfq=(q_d)_{d\geq1}$ are two sequences of polynomials with infinitesimal distributions $(\mu,\mu')$ and $(\nu,\nu')$, respectively, 
what is the infinitesimal distribution of the finite-free additive convolution $\mfp \boxplus \mfq := (p_d\boxplus_d q_d)_{d\geq1}$ and  the finite-free multiplicative convolution $\mfp \boxtimes \mfq := (p_d \boxtimes q_d)_{d\geq1}$?\\

\textbf{Moment-cumulant relations.} 
It is already known from \cite{marcus2021polynomial,arizmendi2018cumulants,arizmendi2023finite} that if $\mfp=(p_d)_{d\geq1}$ and $\mfq=(q_d)_{d\geq1}$ have limiting distributions $\mu$ and $\nu$, respectively, then the finite-free convolutions $\mfp \boxplus \mfq = (p_d\boxplus_d q_d)_{d\geq1}$ and  $\mfp \boxtimes \mfq = (p_d \boxtimes q_d)_{d\geq1}$ converge to the free additive and multiplicative convolution $\mu \boxplus \nu$ and $\mu \boxtimes \nu$ from \cite{voiculescu1991limit}. 
One of our main contributions is to provide explicit formulas for the infinitesimal moments
\begin{equation*}
m_n'(\mfp \boxplus \mfq)
= \lim_{d\to \infty} d \left( m_n(\mu_{p_d \boxplus_d q_d }) - m_n(\mu \boxplus  \nu)  \right)
%
\quad \text{and} \quad 
%
m_n'(\mfp \boxtimes \mfq) = \lim_{d\to \infty} d \left( m_n(\mu_{p_d \boxtimes_d q_d }) - m_n(\mu \boxtimes  \nu)  \right),
\end{equation*}
see Theorem \ref{thm:infinitesimal_additive_finite_free cumulants} and Theorem \ref{prop:moments_for_multiplicative}.   
These infinitesimal moments complete the picture for the infinitesimal distribution of $\mfp \boxplus \mfq$ and $\mfp \boxtimes \mfq$, however, a lot more can be derived from translating our combinatorial formulas into functional relations of formal power series, and vice versa, as we describe in more detail below. 

Let $\mu,\nu \in \MM(\R)$ be probability distributions with finite moments 
 of all orders.
Recall that the \emph{free cumulants} of $\mu$ is the sequence $\freec{1}(\mu), \freec{2}(\mu),\ldots $ recursively defined  through the \emph{moment-cumulant formula} 
\begin{equation}\label{eqn:moment-cumulant-combinatorial_intro}
m_n(\mu) \ \ = \sum_{\pi \in NC(n)} r_{\pi}(\mu) 
\quad \text{with} \quad 
r_\pi (\mu) : = \prod_{V \in \pi} r_{\abs{V}} (\mu) 
\end{equation}
%
with $NC(n)$ the set of all non-crossing partitions of $[n]:= \{1,2,\ldots,n\}$ for every $n \geq 1$, see Section \ref{sec:preliminaries} for more details on partitions.
%
A fundamental feature of free cumulants is that they linearize free additive convolution, that is, 
\begin{equation*}
r_n(\mu \boxplus \nu) = r_n(\mu) +  r_n(\nu)
\end{equation*}
for all $n \geq 1$. 
Moreover, the moment-cumulant formula \eqref{eqn:moment-cumulant-combinatorial_intro} has a simple reformulation in terms of formal power series. 
Indeed, taking the \emph{Cauchy transform} and the \emph{$R$-transform} of $\mu$ as the formal power series 
\begin{equation}\label{eqn:cauchy_and_r_transform} 
G_{\mu}(z)  = 
		\sum_{n=0}^{\infty}   m_n(\mu) \, z^{-n-1} 
\qquad \text{and} \qquad
R_{\mu}(z) =
	\sum_{n=1}^{\infty} r_n(\mu) \, z^{n-1} \ ,
\end{equation}
respectively, 
the moment-cumulant formula \eqref{eqn:moment-cumulant-combinatorial_intro} holds for every $n \geq 1$  if and only if $G_{\mu}$ and $R_{\mu}$ satisfy the functional relation 
\begin{equation}\label{eqn:moment-cumulant-functional_G_R_intro}
\textstyle G_{\mu}\left( R_{\mu}(z)+\frac{1}{z} \right) = z \, 
\end{equation}
with all operations in \eqref{eqn:moment-cumulant-functional_G_R_intro} taken at the level of formal power series. 
Analogously, there exists a notion of cumulants for finite-free convolution, which was introduced in \cite{arizmendi2018cumulants}.  
Concretely, for a polynomial $p_d\in \pols_d(\rr)$, its \emph{finite-free cumulants} is the sequence $\kappa_1(p_d), $ $\kappa_2(p_d),  \ldots, \kappa_d(p_d)$ where each $\kappa_n(p_d)$ is given by  
\begin{equation}\label{eqn:finite-cumulants_coeff_intro}
\kappa_n(p_d) = 
			\frac{(-d)^{n-1}}{(n-1)!} 
			\sum_{\pi \in \partlat(n)} 
			\ (-1)^{\blocks{\pi}-1}
(\blocks{\pi}-1)! \prod_{V\in \pi}  \coef{|V|}{p_d}
\end{equation} 
with $\partlat(n)$ the set of all partitions of $[n]$. 
%
As desired, finite-free cumulants linearize finite-free additive convolution, so given $p_d,q_d\in \pols_d(\rr)$, we have  
\begin{equation*}
\kappa_n(p_d \boxplus_d q_d )=\kappa_n(p_d) + \kappa_n(q_d)
\end{equation*}
for all $n \geq 1$. 
While \eqref{eqn:finite-cumulants_coeff_intro} is a cumulant-coefficient formula, moment-coefficient and moment-cumulant formulas were also provided in \cite{arizmendi2018cumulants}.  
In particular, for some of our results, we depart from \cite[Theorem 1.3]{arizmendi2023finite}, which states that 
\begin{equation}\label{eqn:moment_cumulant_ver_intro}
m_n(\mu_{p_d})  
\ \ \, =  
		\sum_{\pi\in NC(n)} \kappa_\pi(p_d)  
	\ -  \ 
		\frac{n}{2 d} 
		\sum_{ \substack{t,s=n\\  \sigma \in S_{NC}(t,s)}} 
		\frac{\kappa_\sigma(p_d)}{ts} 
	\ \ + \ \ 
		o(d^{-1})
\end{equation}
with $S_{NC}(t,s)$ the set of all non-crossing annular permutations in a $(t,s)$-annulus. \\

\textbf{Infinitesimal dirstributions.} As a consequence of \eqref{eqn:moment_cumulant_ver_intro}, we establish in Lemma \ref{lem:infinitesimal_single_polynomial} that a sequence $\mfp = (p_d)_{d\geq1}$ with limiting distribution $\mu \in \MM(\R)$ possesses infinitesimal moments of all orders $(m'_n(\mfp))_{n\geq 1}$ 
if and only if there exists a sequence $(\widehat{r}_n(\mfp))_{n\geq 1}$ that does not depend on $d$ such that 
\begin{equation}\label{eqn:hypothesis_cumulants_ver_intro}
\kappa_n (p_d)
= r_n (\mu)+ \frac{1}{d} \, \widehat{r}_n (\mfp) + o(d^{-1}) 
\end{equation}
for all $n \geq 1$. 
In other words, the moments $m_n(\mu_{p_d})$ have a truncated expansion in $1/d$ if and only if the finite-free cumulants $\kappa_n (p_d)$ do as well, compare \eqref{eqn:infinitesimal_description} and \eqref{eqn:hypothesis_cumulants_ver_intro}. 
Moreover, the sequences $(m'_n(\mfp))_{n\geq 1}$  and $(\widehat{r}_n(\mfp) )_{n\geq 1}$ determine each other through the relation 
\begin{equation}\label{eqn:mom.cum.inf_intro}
m'_n(\mfp) \ \ \ =\sum_{\pi \in NC(n)} 
            \sum_{V \in \pi } 
                    \widehat{r}_{\abs{V}} (\mfp) \cdot 
                    r_{\pi \setminus V}( \mu )
\ \ - \ \ 
	\frac{n}{2}  \sum_{\substack{ t+s=n\\ \sigma \in S_{NC}(t,s)  } }
					\frac{  r_{\sigma} (\mu) } {ts} 
\end{equation}
for all $n \geq 1$. 
In particular, if a sequence  $\mfp=(p_d)_{d\geq1}$ does have infinitesimal distribution $(\mu,\mu')$, there always exists a sequence $(\widehat{r}_n(\mfp) )_{n\geq 1}$ that satisfies \eqref{eqn:hypothesis_cumulants_ver_intro}. 
In this case, we consider the formal power series 
\begin{equation}\label{eqn:G_mu_prime_R_hat_intro}
G_{\mu'}(z)  \, = 
		\sum_{n=1}^{\infty}   m'_n(\mfp) \, z^{-n-1} 
\qquad \text{and} \qquad 
\widehat{R}_{\mfp}(z) \, =  \sum_{n=1}^{\infty} \widehat{r}_{n}(\mfp) \, z^{n-1} \, .
\end{equation}
Note that if $\mu'$ can be identified with a signed measure on $\R$, then the formal power series $G_{\mu'}(z)$ is simply its Cauchy transform. 
The elements of sequence $(\widehat{r}_n(\mfp))_{n\geq 1}$ are referred to as the \textit{(finite-free) cumulant fluctuations of} $\mfp=(p_d)_{d\geq1}$. 
It is expected that \eqref{eqn:mom.cum.inf_intro} relates real infinitesimal moments from  \cite{cavsbron2025infinitesimal}, arising from orthogonally invariant matrices. 

Turning the combinatorial formula \eqref{eqn:mom.cum.inf_intro} into a functional relation requires the use of a novel formal power series, that we call the \emph{$H$-transform} of $\mu$, defined as 
\begin{equation}\label{eqn:H-transform_intro}
\gdos_{\mu}(z) = \sum_{n=2}^{\infty} h_n (\mu) \, z^{-n-1}  
\quad \text{with} \quad 
h_n (\mu) \ = \, \frac{n}{2} \sum_{\substack{t+s=n \\ \sigma \in S_{NC}(t,s)}} \frac{r_\sigma(\mu)} { ts } \ .
\end{equation}
In Proposition \ref{prop:HG-functional}, we establish that $\gdos_{\mu}(z)$ can be written in terms of $G_{\mu}(z)$ as 
\begin{equation*}
H_\mu (z) = \frac{G'_\mu(z)}{G_\mu(z)}-\frac{G''_\mu (z) }{2G'_\mu (z)} \ .
\end{equation*}
It is worth mentioning that $\gdos_{\mu}(z)$ coincides the Cauchy transform of the possibly signed measure $\tfrac{1}{2} \left(\mathcal{M}(\mu)- \mathcal{M}(\mathcal{M}(\mu))\right)$, where $\mathcal{M}$ is the inverse Markov-Krein transform from \cite{kerov1998}, whenever $\mathcal{M}(\mathcal{M}(\mu))$ is well-defined. 
Then, we prove in Lemma \ref{lemma:additive.delta.zero}  that \eqref{eqn:mom.cum.inf_intro} is equivalent to the following functional relation 
\begin{equation} \label{eq. qis0_intro}
G_{\mu'} = - (\widehat{R}_{\mfp} \circ G_{\mu} ) \cdot G'_{\mu} - \gdos_{\mu} 
\end{equation}
%
%
This recovers  \cite[Theorem 1.6]{arizmendi2023finite} as a particular case when $\widehat{R}_{\mfp}(z)=0$. 

Assume now that both $\mfp=(p_d)_{d\geq1}$ and $\mfq=(q_d)_{d\geq1}$ have infinitesimal distributions, say $(\mu,\mu')$ and $(\nu,\nu')$, respectively. 
The content of Theorem  \ref{thm:infinitesimal_additive_finite_free cumulants}'s proof is that, under these assumptions, the cumulant fluctuations $(\widehat{r}_n( \mfp \boxplus \mfq ) )_{n\geq 1}$ exist, and, in fact, they satisfy 
\begin{equation}\label{eqn:linearity_cumulutant_fluc}
\widehat{r}_n( \mfp \boxplus \mfq ) = \widehat{r}_n( \mfp ) + \widehat{r}_n( \mfq )
\text{ \ for all } n \geq 1,  
\text{ or equivalently \ }  \widehat{R}_{\mfp \boxplus \mfq}(z)= \widehat{R}_{\mfp}(z) + \widehat{R}_{\mfq}(z) \, .
\end{equation}
%
%
%
It is then guaranteed that $\mfp \boxplus \mfq = (p_d\boxplus_d q_d)_{d\geq1}$ has infinitesimal distribution $(\rho,\rho')$ with $\rho=\mu \boxplus \nu$ and  $\rho'$ determined through 
the relation  
\begin{equation}\label{eqn:Gprime_addition_intro}
G_{\rho'} =  
 		-  
				(\widehat{R}_{\mfp} \circ G_{\mu \boxplus \nu} ) 
				\cdot
				G'_{\mu \boxplus \nu}
    	-  
				(\widehat{R}_{\mfq} \circ G_{\mu \boxplus \nu} ) 
				\cdot 
				G'_{\mu \boxplus \nu}
		-
			\gdos_{\mu \boxplus \nu} \, .
\end{equation}
Theorem \ref{thm:infinitesimal_multiplicative_finite_free} establishes that the  cumulant fluctuations $(\widehat{r}_n( \mfp \boxtimes \mfq ) )_{n\geq 1}$ also exist, 
unfortunately, their computation \eqref{eqn:rhats_multiplicative} is quite intricate. 
Nonetheless, their sole existence guarantees that $\mfp \boxtimes \mfq = (p_d\boxtimes_d q_d)_{d\geq1}$ has infinitesimal distribution $(\tau,\tau')$ with $\tau=\mu \boxtimes \nu$ and  $G_{\tau'}(z) =  \sum_{n=1}^{\infty} m'_{n}(\mfp  \boxtimes \mfq) \, z^{-n-1}$ satisfying
\begin{equation}\label{eqn:Gprime_multiplication_intro}
G_{\tau'} =  
		-  
			(\widehat{R}_{\mfp \boxtimes \mfq} \circ G_{\mu \boxtimes \nu} )
			\cdot 
			G'_{\mu \boxtimes \nu}
		-
			\gdos_{\mu \boxtimes \nu} 
\end{equation}
Theorem \ref{thm:infinitesimal_additive_finite_free} and Theorem \ref{prop:cauchy_multiplicative_finite_free} pertain to \eqref{eqn:Gprime_addition_intro} and \eqref{eqn:Gprime_multiplication_intro}, respectively. \\

\textbf{Infinitesimal transforms.} A natural question then relates to the connection between the infinitesimal distribution \eqref{eqn:Gprime_addition_intro} and the notion of infinitesimal freeness from \cite{fevrier2010infinitesimal,belinschi2012free}. 
For this, let us recall that the \emph{infinitesimal free cumulants} of $(\mu,\mu')$, denoted by $(r'_n(\mu,\mu'))_{n\geq1}$, are recursively defined via the \emph{infinitesimal moment-cumulant formula}
\begin{equation*}\label{eqn:infinitesimal-moment-cumulant-formula_intro}
r'_n (\mu,\mu')
		\ \ = 
				\sum_{\pi \in NC(n)} 
				\sum_{V \in \pi } 
					m_{\abs{V}}(\mu') \, 
					m_{\pi \setminus V}(\mu) \, 
					\text{Möb}(\pi,1_n)
\end{equation*}
for every $n \geq 1$. 
This formula can be inverted and, given  $(r_n(\mu))_{n\geq1}$ and $(m_n(\mu))_{n\geq1}$, the sequences $(r'_n (\mu,\mu'))_{n\geq1}$ and $(m_n (\mu'))_{n\geq1}$ determine each other, see Proposition \ref{prop:infinitesimal-moment-cumulant-functional}.
Moreover, the \emph{infinitesimal Cauchy transform} and the \emph{infintesimal $R$-transform} of $(\mu,\mu')$ as formal power series are respectively given by
\begin{equation*}
G_{\mu'}(z) \, = \sum_{n=0}^\infty m_n(\mu') \, z^{-n-1} 
\qquad \text{and} \qquad 
\rinf_{\mu,\mu'}(z) \ =\sum_{n=1}^\infty r'_n (\mu,\mu') \, z^{n-1}. 
\end{equation*}
Furthermore, these formal power series satisfy the relations  
\begin{equation*} 
G_{\mu'} \  = -(\rinf_{\mu,\mu'} \circ G_\mu) \, G'_\mu 
\qquad \text{and} \qquad 
\rinf_{\mu,\mu'} = -(G_{\mu'} \circ K_\mu) \, K_\mu'
\quad \text{where} \quad K_\mu = R_{\mu} + \tfrac{1}{z}.
\end{equation*}
Then, a pair $(\gamma,\gamma')$ is the \emph{infinitesimal free additive convolution} of $(\mu,\mu')$ and $(\nu,\nu')$, denoted $(\mu,\mu')\boxplus_B(\nu,\nu')$, if $\gamma=\mu \boxplus \nu$ and $\gamma'$ has Cauchy transform given by 
\begin{equation}\label{eqn:Cauchy_Rtransforms_infinitesimal_additive}
G_{\gamma'} = - ( \rinf_{\mu,\mu'} \circ G_{\mu \boxplus \nu}  ) \cdot G'_{\mu \boxplus \nu}
- ( \rinf_{\nu,\nu'} \circ G_{\mu \boxplus \nu}  ) \cdot G'_{\mu \boxplus \nu} \, .
\end{equation}
%

\textbf{Relation to infinitesimal freeness.} The discrepancy between \eqref{eqn:Cauchy_Rtransforms_infinitesimal_additive}  and the first two terms in the right-hand side of \eqref{eqn:Gprime_addition_intro} boils down to the deviation of $\widehat{R}_{\mfp}$ and $\widehat{R}_{\mfq}$ from $\rinf_{\mu,\mu'}$ and  $\rinf_{\nu,\nu'}$, respectively. 
In Lemma  \ref{lem:infinitesimal_single_polynomial_3}, we prove that $\widehat{R}_{\mfp}$ and $\rinf_{\mu,\mu'}$ relate to each other through the equation 
\begin{equation}\label{eqn:relation_rinf_rhat}
\widehat{R}^{}_\mfp(z)
=
        \rinf_{\mu,\mu'} (z)-(H_\mu \circ K_\mu) \cdot  K'_\mu.
\end{equation}
%
%
Moreover, since $K_\mu$ is never a constant function, the formal power series $\widehat{R}_{\mfp}$ and $\rinf_{\mu,\mu'}$ are the same if and only if $H_{\mu} = 0$. 
An analogous equation holds for $\widehat{R}_{\mfq}$ and $\rinf_{\nu,\nu'}$. 

A crude relation between $G_{\rho'}$ and $G_{\gamma'}$ can then be derived by plugging \eqref{eqn:relation_rinf_rhat} into \eqref{eqn:Gprime_addition_intro}, and then comparing the result to \eqref{eqn:Cauchy_Rtransforms_infinitesimal_additive}. 
However, as we show in Theorem \ref{prop:fluctuations_infinitesimal_subordination}, a much nicer relation holds when we consider subordination functions, namely, 
\begin{equation}\label{eq:Gprime_subordination_intro}
G_{\rho'} = 
		G_{\gamma'}
	+	
		\gdos_{\mu \boxplus \nu}  
	+
		\frac{\omega''_1}{2\omega'_1}+\frac{\omega''_2}{2\omega'_2} 
\end{equation}
with $\omega_1$ and $\omega_2$ the subordination functions from \cite{voiculescu1993analogues,biane1998processes}, which satisfy the relation 
$G_{\mu\boxplus \nu} = G_\mu\circ\omega_1 = G_\nu\circ \omega_2$, see Section \ref{sec:preliminaries}. 
Additionally, without resorting to the theory of subordination functions, we show in Corollary \ref{cor:additive.delta.zero} that $G_{\rho'}$ and $G_{\gamma'}$ coincide if either  $\mu=\delta_1$ or $\nu=\delta_1$.  A similar result holds for the multiplicative case, as we state in Corollary \ref{Cor. infmult}.
This allows us to recover and generalize the infinitesimal distributions obtained in \cite[Section 4.2]{shlyakhtenko2018free}, which arose from finite-rank deformations of certain random matrices.  

Finally, we apply our results to various concrete situations and examples, obtaining distinct infinitesimal distributions, among which we highlight three of general interest:
\begin{enumerate}[(i)]
    \item The change in infinitesimal distribution after repeated differentiation of polynomials, see Proposition \ref{prop:repeated.diff.general} and Examples \ref{ex.hermite}, \ref{ex.bernoulli}, and \ref{ex.derivatives}.
    \item The infinitesimal distribution from the convolution of polynomials with zero infinitesimal moments, or constant root distribution, see Example \ref{ex.additivefixedmoments}.
    \item  The infinitesimal distribution from the convolution with polynomials of the form $(x-\alpha)^{d-s}q_s$ for a fixed $q_s \in \pols_s(\rr)$ and $\alpha \in \{0,1\}$.
	These correspond to finite-rank perturbations of a multiple of the identity matrix, see Examples  \ref{exm:deviation_from_dirac},  \ref{exm:deviation_from_dirac2}, \ref{exm:two_perturbations_identity}, and \ref{exm:deviation_from_dirac3}. 
\end{enumerate}

\noindent\textbf{Outline of the paper.} 
Besides the introduction, the paper contains four more sections organized as follows. 
In Section \ref{sec:preliminaries}, we provide some preliminaries, this includes notation for partitions and permutations, as well as some notions and results from free convolutions, infinitesimal freeness, and finite-free probability. 
In Section \ref{sec:combinatorial.description}, we provide a combinatorial description for the infinitesimal distribution of finite-free convolutions. 
In Section \ref{Functional}, we translate our combinatorial formulas into functional relations and determine its relation with infinitesimal free convolution. 
Finally, in Section \ref{sec:applications_examples}, we provide several applications and examples. 


\section{Preliminaries}\label{sec:preliminaries}
%
%
%
In this section, we introduce the notation used in this paper regarding set partitions, permutations, and polynomials. 
We also survey some notions and results from free probability and finite-free probability that will be needed for upcoming sections.

\subsection{Partitions and permutations } 

\subsubsection{Set partitions} \label{ssec:partitions}
A \textit{(set) partition} of the set \([n] := \{1, \dots, n\}\) is a collection \(\pi = \{V_1, \dots, V_k\}\) of pairwise disjoint and non-empty subsets \(V_i \subset [n]\) whose union covers all of \([n]\), i.e., $V_i \neq \emptyset$, $V_i \cap V_j = \emptyset$ if $i \neq j$, and $V_1 \cup \dots \cup V_k = [n]$. 
The subsets \(V_i\) are called the \textit{blocks} of \(\pi\), and the number of blocks is denoted by \(\blocks{\pi}\). The set of all partitions of \([n]\) is denoted by \(P(n)\).

A partition \(\pi = \{V_1, \dots, V_r\} \in P(n)\) is called \textit{non-crossing} if for every \(1 \leq a < b < c < d \leq n\) such that \(a, c \in V_i\) and \(b, d \in V_j\), it necessarily follows that \(V_i = V_j\), meaning all four elements belong to the same block. 
The set of all non-crossing partitions of \([n]\) is denoted by \({NC}(n)\). 
We refer the reader to \cite{nica2006lectures} for a detailed exposition on non-crossing partitions in free probability. 

The set $\partlat(n)$ is endowed with the partial order given by reversed refinement, that is, given $\pi, \sigma \in \partlat (n)$, we write $\pi \leq \sigma$ if every block of $\pi$ is contained in a block of $\sigma$. 
Under this partial order, the set $\partlat(n)$ becomes a lattice with minimum $0_n:=\{\{1\},\{2\},\dots,\{n\}\}$, the partition with $n$ blocks, and maximum $1_n:=\{\{1,2\dots,n\}\}$, the partition with only one block. 
We use $\pi\vee \sigma$ to denote the supremum of the set $\{\sigma,\pi\}$ in $\partlat(n)$, called the \emph{join} of $\pi$ and $\sigma$. 

The \emph{Möbius function on the lattice $\partlat(n)$} is the function $\text{Möb}:\partlat^{(2)}(n) \to \R$ with $\partlat^{(2)}(n):= \{ (\pi,\theta) \in \partlat(n) \times \partlat(n) : \pi \leq \theta \}$ given by 
\begin{equation*}
\text{Möb}(0_n, \pi)= (-1)^{n-\blocks{\pi}} \prod_{V\in \pi} (|V|-1)! 
\end{equation*}
for any $\pi \in \partlat(n)$, and extended to a general $(\pi,\theta) \in \partlat^{(2)}(n)$ through the natural lattice isomorphism between the interval $[\pi,\theta]:=\{ \sigma \in \partlat(n) : \pi \leq \sigma \leq \theta \}$ and another interval of the form $[0_k,\rho]$ with $k = \abs{\pi}$ and $\rho \in P(k)$, yielding $\text{Möb}(\pi,\theta) = \text{Möb}(0_k, \rho)$, see \cite[Section 3.10]{stanley2011enumerative}. 

The set $NC(n)$ is also endowed with the partial order given by reversed refinement, making it a sublattice of $\partlat(n)$ with $0_n$ and $1_n$ as minimum and maximum, respectively, as well in this case. 
The \emph{Möbius function on the lattice $NC(n)$}, denoted by $\text{Möb}_{NC}( \cdot , \cdot)$, satisfies 
\begin{equation*}
\text{Möb}_{NC}(0_n, \pi) = (-1)^{n - |\pi|} \prod_{V \in \pi} C_{|V|-1} 
\end{equation*}
with $\pi \in NC(n)$ and  $C_k=\frac{1}{k+1} \binom{2k}{k}$ the $k$-th Catalan number; 
for a general $(\pi,\theta) \in NC^{(2)}(n)$, the value $\text{Möb}_{NC}(\pi,\theta)$ is computed through a canonical interval factorization, see \cite[Lecture 10]{nica2006lectures} for details. 
%
%
%
%
%

%
Let $S_n$ denote the symmetric group on $[n]$. 
For every partition $\pi\in P(n)$, we let $\pi$ also denote the unique permutation in $S_n$ whose cycles are all the blocks of $\pi$ written in increasing order. 
For instance, the partition $\{ \{3,1,5\},\{6,2\},\{4\}\} \in P(6)$ becomes the permutation $(1,3,5)(4)(2,6) \in S_6$ in cycle notation.  
Additionally, for every permutation $\sigma\in S_n$, we let $\sigma$ denote the partition in $P(n)$ whose blocks are the all cycles of $\sigma$ viewed as subsets of $[n]$. 
For example, the permutation $ (6,3,5) (1,4) (2) \in S_6$ is associated with the partition $\{ \{6,3,5\},\{1,4\},\{2\}\} \in P(6)$. 
This is certainly an abuse of notation but one that is rather convenient and unharmful for our purposes.  
%
%
%
%
%

For instance, letting $\gamma_n:=(1,2, \ldots, n)\in S_n$, the set of non-crossing partitions $NC(n)$ is alternatively given by  
\begin{equation}\label{eqn:algebraic_definition_nc}
NC(n) = \{ \pi \in P(n) : |\pi| + |\pi^{-1}\gamma_n|= n+1  \} \, .
\end{equation}
where, for a given $\pi \in P(n)$, we first take $\pi^{-1}$ as the inverse of $\pi$ as a permutation, then we compute the product $\pi^{-1}\gamma_n$ in $S_n$, and finally we consider the permutation $\pi^{-1}\gamma_n$ as a partition in $P(n)$. 
For every non-crossing partition $\pi \in NC(n)$, the permutation $\pi^{-1}\gamma_n$ viewed as a partition is called the \emph{Kreweras complement} of $\pi$ and denoted by $Kr(\pi)$. 
Note that $Kr(\pi)\in NC(n)$ for any $\pi \in NC(n)$ by the defining condition in \eqref{eqn:algebraic_definition_nc}. 
For a geometric definition of the Kreweras complement, see \cite[Definition 9.21]{nica2006lectures}. 
%
%
%
%
%
%

\subsubsection{Non-crossing annular permutations}

The set of \emph{annular non-crossing permutations} $S_{NC}(t,s)$ can be defined as follows. 
First, given integers $t,s \geq 1$, we let $\gamma_{t,s}$ denote the permutation 
$(1,2,\ldots,t)(t+1,t+2\ldots,t+s)  \in S_{t+s}$. 
%
%
Then, for a permutation $\sigma \in S_{t+s}$, we write $\sigma \vee \gamma_{t,s} = 1_{t+s}$ to mean that $\langle \sigma, \gamma_{t,s} \rangle$ ---the subgroup of $S_{t+s}$ generated by $\sigma$ and $\gamma_{t,s}$--- acts transitively on $[t+s]$, i.e., the group action of $\langle \sigma, \gamma_{t,s} \rangle$ on $[t+s]$ has a single orbit. 
Noticeably, the condition  $\sigma \vee \gamma_{t,s} = 1_{t+s}$ is equivalent to $\sigma$ mapping at least one element from $[t]$ to an element in $[t+1,s]$. 
Finally, we define $S_{NC}(t,s)$ as the set of permutations given by 
\begin{equation*}
S_{NC}(t,s) = \Big\{ \sigma \in S_{t+s} \ : \ \gamma_{t,s} \vee \sigma = 1_{t+s} \ \text{ and } \ \abs{\sigma} + \abs{\sigma^{-1} \, \gamma_{t,s}} = t+s \Big\}
\end{equation*}
where $| \cdot |$ gives the number of all distinct cycles in a permutation including fixed points. 
A permutation $\sigma \in S_{NC}(t, s)$ is called \emph{non-crossing in the $(t,s)$-annulus}, or simply \emph{annular non-crossing} if it is clear which pair of integers $(t,s)$ are being considered;  
additionally, the permutation $Kr_{t,s}(\sigma):=\sigma^{-1} \gamma_{t,s}$ is referred to as the \emph{Kreweras complement} of  $\sigma$ in the $(t,s)$-annulus. 

The term \emph{non-crossing} for permutations in $S_{NC}(t, s)$ is due to the fact that their defining conditions are equivalent to the following graphical representation, see Theorem 8 in \cite[Section 5.1]{mingo2017free}. 
For it, one starts with an annulus in the plane with the integers $1$ to $t$ arranged clockwise on the outer circle and the integers $t + 1$ to $t+s$ arranged counterclockwise on the inner circle, this is called a \emph{$(t,s)$-annulus}. 
Then, a permutation $\sigma \in S_{t+s}$ is {non-crossing} precisely if the cycles of $\sigma$ can be represented in such a way that: (i) each cycle of $\sigma$ is drawn as a simple closed path oriented clockwise lying between the inner and outer circles and passing through all the numbers in the given cycle, (ii) the inner and outer circles are connected by at least one cycle, (iii) no cycles cross each other. 

In the context of random matrices, annular non-crossing permutations were first investigated in \cite{mingo2004annular} as the combinatorial object governing the limiting behavior of the covariance of traces of complex Wishart matrices. 
This behavior led to the notion of \emph{second order freeness} \cite{mingo2006second}, and shown to be a common feature for unitarily invariant random matrices in \cite{mingo2007second}. 
This notion does not yet appear in our work, so we simply refer the reader to \cite[Chapter 5]{mingo2017free}, and references therein, for more details on this matter. 
Nevertheless, annular non-crossing permutations do play a role in the combinatorial description of the infinitesimal distribution of Gaussian Orthogonal Ensembles and real Wishart matrices, see  \cite{mingo2019non} and \cite{mingo2025asymptotic}, respectively, and so do they, as we will show later on, in the infinitesimal distribution of finite-free convolutions.

\subsection{Free probability}

\subsubsection{Combinatorial and functional relations}

In this part, we collect some functional relations of formal power series from free probability, as well as their equivalent formulations in terms of their coefficients. 
These relations are interesting on their own, and we will use them in different situations throughout this work. 
Because of that, we have opted to present them separately in a general setting. 

\begin{notation}
Let $(u_n)_{n\geq 1}$ be a sequence of numbers. 
For $\pi\in \partlat(n)$ and $\sigma \in S_n$, we let 
\begin{equation*}
u_{\pi} =\prod_{V\in \pi} u_{|V|} 
\qquad \text{and} \qquad
u_{\sigma}=\prod_{ c \in \sigma} u_{|c|} 
\end{equation*}
where $\abs{ \, \cdot  \, }$ denotes the set cardinality and cycle length, respectively. 
For instance, if $ \pi = \{ \{3,1,5\},\{6,2\},$ $\{4\}\} \in P(6)$ and $\sigma=(3,5) (1,4) (2) \in S_5$, then $u_{\pi} =u_3 \cdot u_2 \cdot u_1$ and $u_{\sigma}=u_2 \cdot u_2 \cdot u_1$.
\end{notation}

We first put together various well-known relations from free independence, commonly phrased in terms of the Cauchy transform and the $R$-transform of a probability distribution. 
Nonetheless, they hold in the general setting of formal power series. 
Proofs for these relations can be found in both \cite[Lecture 12]{nica2006lectures} and \cite[Section 2.4]{mingo2017free}.  

\begin{proposition}\label{prop:basic.funct.relation}
Let $(m_n)_{n \geq 1}$ and $(r_n)_{n \geq 1}$ be any two sequences of numbers and consider the formal power series  $G(z)  = z^{-1} + \sum_{n=1}^{\infty}  m_n \, z^{-n-1} $ and $K(z) = z^{-1} + \sum_{n=1}^{\infty} r_n \, z^{n-1} $. 
The following are equivalent: 
\begin{enumerate}[(1)]
\item The functional relation $ K ( G  (z) ) = z$ holds at the level of formal power series. 
\item The functional relation $ G ( K  (z) ) = z$ holds at the level of formal power series. 
\item The combinatorial formula $ m_n = \sum_{\pi \in NC(n)} r_{\pi}$ holds for every $n \geq 1$. 
\item The combinatorial formula $ r_n = \sum_{\pi \in NC(n)} m_{\pi} \, \text{Möb}_{NC}(\pi,1_n)$ holds for every $n \geq 1$. 
\end{enumerate}
Conditions (1) and (2) can also be formulated in terms of the formal power series $R(z) =  \sum_{n=1}^{\infty} r_n \, z^{n-1} $. 
\end{proposition}

We now gather the analogue relations for infinitesimal free independence. 
Early versions of some of these can be tracked down to \cite{biane2003non}, with later variations and refinements established in \cite{fevrier2010infinitesimal, belinschi2012free, mingo2019non}. 
Here we provide condensed yet complete proofs based on those from the aforementioned works. 
\begin{proposition} \label{prop:infinitesimal-moment-cumulant-functional}
Let $\{m'_n\}_{n \geq 1}$ and $\{r'_n\}_{n \geq 1}$ be any two sequences of numbers and consider the formal power series  $g(z)  = \sum_{n=1}^{\infty}  m'_n \, z^{-n-1} $ and $r(z) = \sum_{n=1}^{\infty} r'_n \, z^{n-1}$. 
Suppose that $G(z)$ and $K(z)$ satisfy any of the four equivalent conditions from Proposition \ref{prop:basic.funct.relation}. 
The following are equivalent: 
\begin{enumerate}[(1')]
\item The functional relation $g(z) = -r(G(z)) \cdot G'(z)$ holds at the level of formal power series. 
\item The functional relation $r(z) = -g(K(z)) \cdot K'(z)$ holds at the level of formal power series. 
\item The combinatorial formula $m'_n  = \sum_{\pi \in NC(n)} \sum_{V \in \pi } r'_{\abs{V}} \cdot r_{\pi \setminus V}$ holds for every $n \geq 1$. 
\item The combinatorial formula $r'_n  = \sum_{\pi \in NC(n)} \sum_{V \in \pi } m'_{\abs{V}} \, m_{\pi \setminus V} \, \text{Möb}_{NC}(\pi,1_n)$ holds for every $n \geq 1$. 
\end{enumerate}
\end{proposition}

\begin{proof}
From Proposition \ref{prop:basic.funct.relation}, we know that $G(K(z))=z$, and consequently  $G'(K(z))\cdot K'(z) = 1$.
Then, assuming \textit{(1')} holds, and replacing $z$ with $K(z)$, we obtain $g(K(z)) = -r(z) \cdot G'(K(z))$, which is equivalent to \textit{(2')}. 
Hence, \textit{(1')} implies \textit{(2')}.  
A similar argument gives the reversed implication. 

The equivalence between \textit{(1')} and  \textit{(3')} goes as follows. 
Letting $m_0 = 1$, and writing $G(z)  = \sum_{n=0}^{\infty}  m_n \, z^{-n-1} $,  direct computations give 
\begin{equation}\label{eqn:equiv_1-3_pre}
-r(G(z)) \cdot G'(z) = \sum_{n=1}^{\infty}
					\left(
							\sum_{k=1}^{n} \sum_{\substack{ n_1,\ldots,n_k=0 \\ n_1+\cdots +n_k+k=n}}^{n-k}
								(n_k+1) \cdot r'_k \cdot
								m_{n_1} \cdots m_{n_k}
					\right) z^{-n-1}.
\end{equation}
Then, using the moment-cumulant formula \textit{(3)} from Proposition \ref{prop:basic.funct.relation} for $n_1,\ldots,n_k$, an enumerative argument yields 
\begin{equation}\label{eqn:equiv_1-3}
	\sum_{k=1}^{n} \sum_{\substack{ n_1,\ldots,n_k=0 \\ n_1+\cdots +n_k+k=n}}^{n-k}
								(n_k+1) \cdot r'_k \cdot
								m_{n_1} \cdots m_{n_k}
=
\sum_{\pi \in NC(n)} \sum_{V \in \pi } r'_{\abs{V}} \cdot r_{\pi \setminus V} \ .
\end{equation}
Indeed, writing $V=\{v_1 < v_2 < \cdots < v_k \}$ for each block $V$ in the right-hand side of \eqref{eqn:equiv_1-3}, each product $r'_{\abs{V}} \cdot r_{\pi \setminus V}$ corresponds to $r'_k \cdot  r_{\pi_1}   r_{\pi_2}  \cdots  r_{\pi_k}$ for some non-crossing partitions $\pi_1 \in NC(n_1), \pi_2 \in NC(n_2), \ldots, \pi_k \in NC(n_k)$ with $n_i$ the number of elements strictly between $v_i$ and $v_{i+1}$ cyclically, i.e., $n_i = v_{i+1}-v_i-1$ for $i=1,2,\ldots,k-1$ and $n_k = n+v_1-v_k-1$, the factor $(n_k+1)$ amounts to the different choices $V$ associated to the same factor $r'_k \cdot  r_{\pi_1}   r_{\pi_2}  \cdots  r_{\pi_k}$. 
I follows from \eqref{eqn:equiv_1-3_pre} and \eqref{eqn:equiv_1-3} that \textit{(1')} and \textit{(3')} are equivalent. 

Finally, the equivalence between \textit{(3')} and  \textit{(4')} is due to that of \textit{(3)} and \textit{(4)} from Proposition \ref{prop:basic.funct.relation}. 
%
%
Indeed, \textit{(3)} and \textit{(4)} remain equivalent for sequences taken over the two-dimensional Grasmann algebra $\CC = \{ ( \begin{smallmatrix} a & a' \\ 0 & a \end{smallmatrix} )  : a,a' \in \C \} \subset \mathrm{Mat}_2(\C)$. 
Thus, letting $\widetilde{m}_n = ( \begin{smallmatrix} m_n & m'_n \\ 0 & m_n \end{smallmatrix} )$ and $\widetilde{r}_n = ( \begin{smallmatrix} r_n & r'_n \\ 0 & r_n \end{smallmatrix} )$, we have that $\widetilde{m}_n = \sum_{\pi \in NC(n)} \widetilde{r}_{\pi}$ if and only if $ \widetilde{r}_n = \sum_{\pi \in NC(n)} \widetilde{m}_{\pi} \, \text{Möb}_{NC}(\pi,1_n)$. 
But elementary matrix operations give $\widetilde{r}_{\pi} =  ( \begin{smallmatrix} r_{\pi} & \sum_{V \in \pi } r'_{\abs{V}} \cdot r_{\pi \setminus V} \\ 0 &  r_{\pi}  \end{smallmatrix} ) $ and $\widetilde{m}_{\pi} = ( \begin{smallmatrix} m_{\pi} & \sum_{V \in \pi } m'_{\abs{V}} \cdot m_{\pi \setminus V} \\ 0 &  m_{\pi}  \end{smallmatrix} )$. 
Therefore, the desired conclusion follows since $m_n = \sum_{\pi \in NC(n)} r_{\pi}$ and $ r_n = \sum_{\pi \in NC(n)} m_{\pi} \, \text{Möb}_{NC}(\pi,1_n)$ hold already by hypothesis. 
\end{proof}

\subsubsection{Free convolutions} \label{ssubsec:free_convolutions}

We will only consider distributions determined by their moments, or equivalently, by their free cumulants, in this work. 
Therefore, the following combinatorial definition of $\boxplus$, in terms of free cumulants, is enough for our purposes. 
Let us recall that, equivalently to the moment-cumulant formula \eqref{eqn:moment-cumulant-combinatorial_intro}, the \emph{free cumulants} of a distribution $\mu \in \MM(\rr)$  with finite moments of all orders can be defined as the sequence $(\freec{n}(\mu))_{n\geq 1}$ given for each $n$ by 
\begin{equation*}
r_n(\mu) \ \ = \sum_{\pi \in NC(n)} m_{\pi} (\mu) \, \text{Möb}_{NC}(\pi,1_n)  .
\end{equation*}

\begin{definition}[Free additive convolution]
For probability distributions $\mu,\nu \in \MM(\R)$ with finite moments of all orders, 
their \emph{free additive convolution} is the unique probability measure $\mu\boxplus \nu$ such that 
\begin{equation*}
r_n(\mu \boxplus \nu) = r_n(\mu) +  r_n(\nu) 
\quad \forall \ n \geq 1.
\end{equation*} 
\end{definition}
The free additive convolution $\boxplus$ is a commutative operation on $\MM(\rr)$, the set of all probability distributions on $\rr$. 
It was first defined for compactly supported distributions in 
\cite{voiculescu1986addition}, as the spectral distribution of the sum of freely independent bounded operators, and then extended to arbitrary distributions in \cite{maassen1992addition,bercovici1993free}, using the reciprocal Cauchy transform and the theory of unbounded operators affiliated with a von Neumann algebra, respectively. 

On the other hand, the free multiplicative convolution $\boxtimes$ is an operation that requires at least one of the factors to be in $\MM(\R_{\geq 0})$, the set of all probability distributions with support on the non-negative real numbers. 
It was first defined in 
\cite{voiculescu1987multiplication} and then extended to full generality in \cite{bercovici1993free}. 
As mentioned above, a combinatorial definition of $\boxtimes$ is enough for us. 

\begin{definition}[Free multiplicative convolution]
For probability distributions $\mu \in \MM(\R)$ and $\nu \in \MM(\R_{\geq 0})$ with finite moments of all orders,  
their \emph{free multiplicative convolution} is the unique probability measure $\mu \boxtimes \nu$  such that 
\begin{equation*}
r_n(\mu \boxtimes \nu)
	\ \	=
			\sum_{\substack{ \pi \in NC(n) }}  
					r_\pi(\mu) \,
					r_{Kr(\pi)}(\nu) 
\quad \forall \ n \geq 1.
\end{equation*}
\end{definition}
%
%
%
%

\subsubsection{Infinitesimal free convolutions} \label{ssubsec:inf_free_convolutions}

The infinitesimal free additive convolution $\boxplus_B$ is a commutative operation defined on pairs $(\mu,\mu')$ where $\mu \in \MM(\R)$ has moments of all orders and $\mu':\R[x] \to \R$ is a linear functional with $\mu'(1)=0$, referred to as infinitesimal distributions. 
This operation was introduced along with the notion of infinitesimal freeness in \cite{belinschi2012free}, where various scenarios for when the second coordinates $\mu'$ can be consistently identified with signed measures are established. 
As in the previous part, turning to a combinatorial definition of $\boxplus_B$ suffices for the present work.
%
%
%
%
%

\begin{notation}
For any linear functional $\mu':\R[x] \to \R$, 
we let $m_n(\mu')$ denote the value $\mu'(x^n)$ for each $n \geq 0$. 
In particular, if $\mu'$ is identified with a signed measure on $\R$, each $m_n(\mu')$ corresponds to the $n$-th moment of $\mu'$, namely, $\int_{\rr} t^n  \,  d \mu'(t)$. 
\end{notation}

The \emph{infinitesimal free cumulants} of $(\mu,\mu')$ is the sequence  $(r'_n(\mu,\mu'))_{n\geq 1}$ given by the cumulant-moment formula 
\begin{equation*}\label{eqn:infinitesimal-cumulant-moment-formula}
r'_n (\mu,\mu')\ \ = \sum_{\pi \in NC(n)} \sum_{V \in \pi } m_{\abs{V}}(\mu') \, m_{\pi \setminus V}(\mu) \, \text{Möb}(\pi,1_n)
\quad \forall \ n \geq 1.
\end{equation*}
By Proposition \ref{prop:infinitesimal-moment-cumulant-functional}, this formula can be inverted in a way that the infinitesimal free cumulants can also be defined recursively through the moment-cumulant formula 
\begin{equation*}\label{eqn:infinitesimal-moment-cumulant-formula}
m_n (\mu')\ \ = \sum_{\pi \in NC(n)} \sum_{V \in \pi } r'_{\abs{V}}(\mu,\mu') \, r_{\pi \setminus V}(\mu) 
\quad \forall \ n \geq 1.
\end{equation*}
%
%
It is immediate that the infinitesimal free cumulants $(r'_n (\mu,\mu'))_{n\geq 1}$ and the infinitesimal moments $(m_n (\mu'))_{n \geq 1}$ completely determine each other, in view of Proposition \ref{prop:infinitesimal-moment-cumulant-functional}, and since the same holds for $(r_n(\mu))_{n\geq 1}$ and $(m_n(\mu))_{n\geq 1}$. 
With this in hand, we provide the following. 

\begin{definition}[Infinitesimal free additive convolution] \label{def:infinitesimal_additive_convolution}
For infinitesimal distributions $(\mu,\mu')$ and $(\nu,\nu')$, their \emph{infinitesimal free additive convolution}  $(\mu,\mu')\boxplus_B (\nu,\nu')$ is defined as the pair $(\gamma,\gamma')$ where $\gamma=\mu\boxplus\nu$ and  $\gamma':\R[x] \to \R$ is the unique linear functional with $\gamma'(1)=0$ such that 
\begin{equation}\label{eqn:infinitesimal_cumulants_addition}
r'_n(\gamma,\gamma') = r'_n(\mu,\mu')+r'_n(\nu,\nu') 
\quad \forall \ n \geq 1. 
\end{equation}
\end{definition}

The \emph{infinitesimal Cauchy transform} and the \emph{infintesimal $R$-transform} of $(\mu,\mu')$ as formal power series are respectively given by
\begin{equation}\label{eqn:infinitesimal_cauchy_and_r}
G_{\mu'}(z) \, = \sum_{n=0}^\infty m_n(\mu') \, z^{-n-1} 
\qquad \text{and} \qquad 
\rinf_{\mu,\mu'}(z) \ =\sum_{n=1}^\infty r'_n (\mu,\mu') \, z^{n-1}. 
\end{equation}
Due to Proposition \ref{prop:infinitesimal-moment-cumulant-functional}, these formal power series satisfy the relations  
\begin{equation} \label{eqn:Mingo.relation}
G_{\mu'} \  = -(\rinf_{\mu,\mu'} \circ G_\mu) \, G'_\mu 
\qquad \text{and} \qquad 
\rinf_{\mu,\mu'} = -(G_{\mu'} \circ K_\mu) \, K_\mu'.
\end{equation}
Notice that  \eqref{eqn:infinitesimal_cumulants_addition} written in terms of infintesimal $R$-transforms becomes 
\begin{equation*}
\rinf_{\gamma,\gamma'}(z) = \rinf_{\mu,\mu'}(z) + \rinf_{\nu,\nu'}(z) \, .
\end{equation*}
It follows that a pair $(\gamma,\gamma')$ is the infinitesimal free additive convolution $ (\mu,\mu')\boxplus_B (\nu,\nu')$ if and only if $\gamma = \mu \boxplus \nu$ and $\gamma'$ has Cauchy transform $G_{\gamma'}$ given by 
\begin{equation}\label{eqn:Cauchy_Rtransforms_infinitesimal_additive_preli}
G_{\gamma'} = - ( \rinf_{\mu,\mu'} \circ G_{\mu \boxplus \nu}  ) \cdot G'_{\mu \boxplus \nu}
- ( \rinf_{\nu,\nu'} \circ G_{\mu \boxplus \nu}  ) \cdot G'_{\mu \boxplus \nu} . 
\end{equation}
By  Proposition \ref{prop:infinitesimal-moment-cumulant-functional}, the latter is equivalent to the infinitesimal moments $(m_n (\gamma'))_{n \geq 1}$ satisfying 
\begin{equation*}
m_n (\gamma') 
		= 
			\sum_{\pi \in NC(n)} 
			\sum_{V \in \pi } 	
					\left( r'_{\abs{V}}(\mu,\mu')+r'_{\abs{V}}(\nu,\nu') \right) 
				\cdot 
					r_{\pi \setminus V} (\mu \boxplus \nu)
\quad \forall \ n \geq 1. 
\end{equation*}
We now give a combinatorial definition of the infinitesimal free multiplicative convolution $\boxtimes_B$. 

\begin{definition}[Infinitesimal free multiplicative convolution]
 \label{def:infinitesimal_multiplicative_convolution}
For infinitesimal distributions $(\mu,\mu')$ and $(\nu,\nu')$ with  $\nu \in \MM(\R_{\geq 0})$, their \emph{infinitesimal free multiplicative convolution}  $(\mu,\mu')\boxtimes_B (\nu,\nu')$ is defined as the pair $(\gamma,\gamma')$ where $\gamma = \mu\boxtimes\nu$ and  $\gamma':\R[x] \to \R$ is the unique linear functional with $\gamma'(1)=0$ such that 
\begin{equation*}
r'_n(\gamma,\gamma') 
	\ \ =  
			\sum_{\pi\in NC(n)}\left(
					\sum_{V\in\pi} 
						r'_n(\mu,\mu')  r_{\pi\backslash V}(\mu)
						\cdot 
						r_{Kr(\pi)}(\nu) 
		+
			\sum_{W\in Kr(\pi)}
					r_\pi(\mu)  
					\cdot
					r'_n(\nu,\nu')  r_{Kr(\pi)\backslash W}(\nu)
			\right)
\quad \forall \ n \geq 1. 
\end{equation*}
\end{definition}
In terms of moments, see \cite[Proposition 5.1]{fevrier2010infinitesimal}, the last expression is equivalent to 
\begin{equation}\label{eqn:infinitesimal_moments_of_a_product}
m_n(\gamma') 
  \ \  = 
        \sum_{\pi\in NC(n)}\left(
            \sum_{V\in\pi} 
                m_{|V|}(\mu') \, m_{\pi\backslash V}(\mu) 
                \cdot 
                r_{Kr(\pi)}(\nu) 
    +
        \sum_{W\in Kr(\pi)}
                m_\pi(\mu)
                \cdot 
                r'_{\abs{W}}(\nu,\nu') \,  r_{Kr(\pi)\backslash W}(\nu) \right)
\quad \forall \ n \geq 1. 
\end{equation}
%

\subsubsection{Subordination} 

In some instances, namely, Theorem \ref{prop:fluctuations_infinitesimal_subordination} and Corollary \ref{Cor. infmult}, we will use the analytical approach to infinitesimal convolutions via subordination functions from \cite{belinschi2012free}. 
For this, let us recall that the \emph{analytical Cauchy transform} of a possibly signed measure $\mu$ on $\R$ is given by 
\begin{equation}\label{Cauchy} 
\mathcal{G}_\mu(z)  =\int_\rr\frac{1}{z-t} \,  d \mu(t) 
\quad \forall \ z \in \C \setminus \R.   
\end{equation}
This is an analytical map that takes the complex upper half-plane $\C^+ = \{ z \in \C : \mathfrak{Im}(z) > 0 \}$ to the lower half-plane  $\C^-= \{ z \in \C : \mathfrak{Im}(z) < 0 \}$, and vice versa.  
Moreover, when $\mu$ has compact support, $\mathcal{G}_\mu(z)$ has expansion near infinity, in the sense that 
\begin{equation*} 
\mathcal{G}_\mu(z)  = 
		\sum_{n=0}^{\infty}   m_n(\mu) \, z^{-n-1} 
\end{equation*}
for $\abs{z}$ large enough. 
If, in addition, $\mu$ is a probability distribution, then there exists an analytical map $\mathcal{R}_\mu(z)$, called the \emph{analytical $R$-transform} of $\mu$, defined on a punctured neighborhood of $0$ through the relation
\begin{equation*}
\mathcal{G}_\mu \left( \mathcal{R}_\mu(z) + \tfrac{1}{z} \right) = z  .
\end{equation*}
See \cite[Theorem 3.18]{mingo2017free} for more details. 
It is known from \cite{voiculescu1993analogues,biane1998processes} that for any two probability distributions $\mu, \nu \in \MM(\R)$, there exist unique analytic maps $\omega_1,\omega_2:\cc^+\to\cc^+$ such that
\begin{equation} \label{eqn:def.subordination}
\GG_{\mu\boxplus \nu}(z)
		=
			 \GG_\mu\circ \omega_1(z) 
		=
			\GG_\nu\circ \omega_2(z)
		=
			\frac{1}{\omega_1(z)+\omega_2(z)-z}    
\qquad \forall \ z \in \C^+.   
\end{equation}
The maps $\omega_1$ and $\omega_2$ are referred to as \emph{subordination functions}, and next we provide their relation to infinitesimal convolution. 

\begin{proposition}[{\cite[Theorem 26]{belinschi2012free}}]
\label{prop:infintesimal.subordination}
If $(\gamma,\gamma')=(\mu,\mu')\boxplus_B (\nu,\nu')$, then
\begin{equation*}
\GG_{\gamma'}
		=
				(\GG_{\mu'}\circ\omega_1) \, \omega'_1
			+
				(\GG_{\nu'}\circ\omega_2) \, \omega'_2.
\end{equation*} 
\end{proposition}

Similarly, it is known from \cite{biane1998processes} that there exist subordination functions for the multiplicative convolution. 
In this case, we must consider the \emph{moment generating function} of a possibly signed measure $\mu$ on $\R$ given by 
\begin{equation*}
\psi_{\mu}(z) = \int_{\mathbb{R}} \frac{zt}{1 - zt}  \, d \mu(t)
\quad \forall \ z \in \C \setminus \R.   
\end{equation*}
Then, for any two probability distributions $\mu, \nu \in \MM(\R_{\geq 0})$, there exist unique analytic maps $\omega_1,\omega_2:\C \setminus \R  \to  \C \setminus \R$ such that
\begin{equation} \label{MultSub} 
\psi_{\mu \boxtimes \nu} (z)  
		= 
				\psi_{\mu}  \circ \omega_1(z)  
		= 
				\psi_{\nu}  \circ \omega_2(z)  
\quad \text{ and} \quad
\frac{ \psi_{\mu_1 \boxtimes \mu_2}(z)}{1 + \psi_{\mu_1 \boxtimes \mu_2}(z)}
			=
		\frac{\omega_1(z) \omega_2(z)}{z}    
\quad \forall \ z \in \C \setminus \R.   
\end{equation}
\begin{proposition}[Reformulation of {\cite[Proposition 43]{belinschi2012free}}]
\label{prop:infintesimal.subordinationmulti}
If $(\gamma,\gamma')=(\mu,\mu')\boxtimes_B (\nu,\nu')$, then
\begin{equation}\label{eqn:multiplicative_infinitesimal_cauchy}
\GG_{\gamma'}=(\GG_{\mu'} \circ\tau_1)\tau'_1+(\GG_{\nu'}\circ\tau_2)\tau'_2,\end{equation}
where $ \tau_{i} (z)  = \frac{1}{\omega_{i}(1/z)}.$ 
\end{proposition}
\begin{proof} 
Proposition 43 in \cite{belinschi2012free}, states that 
\begin{equation*}
\frac{\psi_{\gamma'} (z)}{z}
        =
\frac{\psi_{\mu'} (\omega_1(z))}{\omega_1(z)} \omega_1'(z) 
+ 
\frac{\psi_{\nu'} (\omega_2(z))}{\omega_2(z)} \omega_2'(z),
\end{equation*}
which translates into 
\begin{equation}\label{eqn:multiplicative_infinitesimal_psi}
\GG_{\gamma'}(z) = \tfrac{1}{z} \psi_{\gamma'}\left( \tfrac{1}{z}\right)
=
\frac{1}{z^2} 
\frac{\psi_{\mu'} (\omega_1(1/z))}{\omega_1(1/z)} \cdot \omega_1'(1/z) 
+ 
\frac{1}{z^2}  \frac{\psi_{\nu'} (\omega_2(1/z))}{\omega_2(1/z)} \cdot \omega_2'(1/z) \, .
\end{equation}
The desired relation  follows from the identities
\begin{equation*}
 \tau'_{i} (z)   
 =\frac{\omega'_{i}(1/z)}{(z \cdot \omega_i(1/z))^2}
\text{\qquad and \qquad}
\omega_i(1/z) \cdot \psi_{\mu_i'} (\omega_{i}(1/z)) 
=
 \GG_{\mu_i'} \left( \tau_{i}(z) \right).
\end{equation*}
\end{proof}



\subsection{Finite free probability}

In this part, we collect some results on finite-free probability from \cite{arizmendi2018cumulants,arizmendi2023finite} that are used next. 
To this end, recall that any polynomial $p_d\in \pols_d(\rr) $, the set of all monic and real-rooted polynomials in $x$ of degree $d$, can be uniquely written as
\begin{equation*}
p_d(x)=\sum_{k=0}^d x^{d-k}(-1)^k \binom{d}{k} \coef{k}{p_d} 
\end{equation*}
for some coefficients $\coef{k}{p_d} \in \rr$.
Moreover, the \emph{finite-free additive convolution} of $p_d,q_d\in \pols_d(\rr) $ is the polynomial $p_d\boxplus_d q_d \in \mathbb{R}_d[x]$ given by 
\begin{equation*}
[p_d \boxplus_d q_d](x)
	=
		\sum_{k=0}^d x^{d-k}(-1)^k 
				\binom{d}{k} 
				\left[
				\sum_{i=0}^{k}\binom{k}{i} \coef{i}{p_d} \coef{k-i}{q_d}\right] ; 
\end{equation*}
while their \emph{finite-free multiplicative convolution} is the polynomial $p_d\boxtimes_d q_d \in \mathbb{R}_d[x]$ given by 
\begin{equation*}
[p_d \boxtimes_d q_d](x) =
			\sum_{k=0}^d x^{d-k}(-1)^k \binom{d}{k} \Big[ \coef{k}{p_d} \coef{k}{q_d} \Big].
\end{equation*}
\begin{notation}
The \emph{root distribution} of $p_d\in \pols_d(\rr)$ is the discrete probability distribution $\mu_{p_d}$ that puts a mass of $1/d$ at each of its $d$ roots counting multiplicities, written 
$\lambda_d(p_d)  \leq 
\lambda_{d-1}(p_d) \leq \cdots \leq
\lambda_{1}(p_d) $.   
To ease notation, for each $n \geq 0$, we let $m_n(p_d)$ denote the $n$-th moment of $\mu_{p_d}$, that is, 
\begin{equation*} 
m_n(p_d) =\frac{1}{d}\sum_{i=1}^d  ( \lambda_{i}(p_d) )^n
\end{equation*}
\end{notation}

 The \emph{finite-free cumulants} of each $p_d\in \pols_d(\rr)$ is the sequence $\kappa_1(p_d), $ $\kappa_2(p_d),  \ldots, \kappa_d(p_d)$ given by 
\begin{equation}\label{eqn:finite-cumulants_coeff_preli}
\kappa_n(p_d) = 
			\frac{(-d)^{n-1}}{(n-1)!} 
			\sum_{\pi \in \partlat(n)} 
				\text{Möb}(\pi,1_n)
				\, \coef{\pi}{p_d}
\end{equation} 
with $\coef{\pi}{p_d} : =  \prod_{V\in \pi} \coef{\blocks{\pi}}{p_d}$ for $\pi \in \partlat(n)$ and $n = 1, 2 , \ldots, d$. 
One of the main properties of $\kappa_n( \cdot)$ is the linearization of the finite-free additive convolution $\boxplus_d$.
\begin{proposition}[{\cite[Proposition 3.6]{arizmendi2018cumulants}}]
For any $p_d,q_d \in \pols_d(\rr) $ and  $1 \leq n \leq d$, we have that 
\begin{equation*}
\kappa_n(p_d \boxplus_d q_d)=\kappa_n(p_d) + \kappa_n(q_d)
\end{equation*}
\end{proposition}

The cumulant-coefficient formula \eqref{eqn:finite-cumulants_coeff_preli} was introduced in \cite{arizmendi2018cumulants}, among other relations describing how the sequences  $(m_n(p_d))_{n\geq 1}$, $(\kappa_n(p_d))_{n=1}^{d}$,  and $(\coef{n}{ p_d})_{n=1}^{d}$ determine each other explicitly. 
In particular, it was proved in \cite[Corollary 4.6]{arizmendi2018cumulants} that the following cumulant-moment formula holds  
\begin{equation} \label{eq.finitemomentcumulant}    
\kappa_n(p_d)
	=	
		\frac{(-1)^{n-1}}{(n-1)!} 
		\sum_{\substack{ \pi,\theta \in P(n) \\  \pi \lor \theta =1_n }}
				d^{\blocks{\pi}+\blocks{\theta}-n-1}
				\text{Möb}(0_n,\pi)
				\text{Möb}(0_n,\theta) 
				m_\pi(p_d)  \, .
\end{equation}
%
%
%
%
%
%
%

The next theorem, concerning finite-free multiplicative convolution, is, in a certain way, a departing point of our work. 

\begin{theorem}[{\cite[Theorem 1.1]{arizmendi2023finite}}]  \label{thm.cumulant.of.products}
For any $p_d,q_d \in \pols_d(\rr) $ and  $1 \leq n \leq d$, we have that 
\begin{equation}\label{eq.cumulant.of.products}
\kappa_n(p_d \boxtimes_d q_d )	
		=
			\frac{(-1)^{n-1}}{(n-1)!} 
			\sum_{\substack{ \pi, \theta\in P(n) \\ \pi\lor  \theta=1_n }} 
					d^{\blocks{\pi}+\blocks{\theta}-n-1}
					\text{Möb}(0_n,\pi)
					\text{Möb}(0_n,\theta) 
					\kappa_\pi(p_d )
					\kappa_\theta(q_d ) 
\end{equation}
and
\begin{equation} \label{eq:multiplicativemomentcumulant}
m_n(p_d \boxtimes_d q_d)
	=
		\frac{(-1)^{n-1}}{(n-1)!} 
		\sum_{\substack{ \pi, \theta\in P(n) \\  \pi\lor \theta=1_n }} 
				d^{\blocks{\pi}+\blocks{\theta}-n-1}
				\text{Möb}(0_n,\pi)
				\text{Möb}(0_n,\theta) 
				\kappa_\pi(p_d) 
				m_\theta(q_d ) \, .
\end{equation}
\end{theorem}
It was also shown in \cite[Lemma 3.3, Lemma 4.4]{arizmendi2023finite} that for any two sequences of numbers $(u_n)_{n \geq 1 }$ and $(v_n)_{n \geq 1 }$ we have 
\begin{equation*}
\sum_{\pi\in NC(n)} 
				u_\pi \,
				v_{Kr(\pi)}
=
\frac{(-1)^{n-1}}{(n-1)!} 
		\sum_{\substack{ \pi, \theta\in P(n) \\  \pi \lor \theta=1_n \\ \blocks{\pi}+\blocks{\theta}= n+1 }  } 
				\text{Möb}(0_n,\pi)
				\text{Möb}(0_n,\theta) 
				u_\pi
				v_\theta
\end{equation*}
and
\begin{equation*}
-	\frac{n}{2d}
	\sum_{\substack{ t+s = n \\ \sigma \in S_{NC}(t,s) } } 
		\frac{u_{\sigma} \, v_{Kr_{t,s}(\sigma)}}{ts}
=
\frac{(-1)^{n-1}}{(n-1)!} 
		\sum_{\substack{ \pi, \theta\in P(n) \\  \pi \lor \theta=1_n \\ \blocks{\pi}+\blocks{\theta}= n}  } 
				\text{Möb}(0_n,\pi)
				\text{Möb}(0_n,\theta) 
				u_\pi
				v_\theta 
\end{equation*}
for any $n \geq 1$. 
And it turns out that all the partitions $\pi,\theta \in P(n)$ with $\pi \lor \theta = 1_n$ and $\abs{\pi} + \abs{\theta} > n + 1$ yield a net zero contribution to both \eqref{eq.cumulant.of.products} and \eqref{eq:multiplicativemomentcumulant}. 
Thus, as a consequence of Theorem \ref{thm.cumulant.of.products} we obtain this. 
\begin{corollary}\label{cor:truncated_expasions}
For any $p_d,q_d \in \pols_d(\rr) $ and  $1 \leq n \leq d$, we have that 
\begin{equation}\label{eqn:truncated_expasion_multiplicative_cumulants}
\kappa_n (p_d\boxtimes_d q_d)
	\ \ 	= 
			\sum_{\pi\in NC(n)} 
				\kappa_\pi(p_d) \,
				\kappa_{Kr(\pi)}(q_d) 
		 \	- \
			\frac{n}{2d}
			\sum_{\substack{ t+s = n \\ \sigma \in S_{NC}(t,s) } } 
				\frac{\kappa_{\sigma}(p_d) \, \kappa_{Kr_{t,s}(\sigma)}(q_d)}{ts} 
		\ \	+ \ \
			o\left(1/d\right)
\end{equation}
and
\begin{equation}\label{eqn:truncated_expasion_multiplicative_moments}
m_n(p_d \boxtimes_d q_d ) 
	\ \ 	= 
			\sum_{\pi\in NC(n)} 
				\kappa_\pi(p_d) \,
				m_{Kr(\pi)}(q_d) 
		\	- \
			\frac{n}{2d}
			\sum_{\substack{ t+s = n \\ \sigma \in S_{NC}(t,s) } } 
				\frac{\kappa_{\sigma}(p_d) \, m_{Kr_{t,s}(\sigma)}(q_d)}{ts} 
		\ \ + \ \
			o\left(1/d\right)  .
\end{equation}
\end{corollary}
While not explicitly stated, the previous two relations are part of \cite[Proof of Theorem 1.3]{arizmendi2023finite}. 
Additionally, since finite-free multiplicative convolution is commutative, then \eqref{eqn:truncated_expasion_multiplicative_moments} is equivalent to 
\begin{equation}\label{eqn:truncated_expasion_multiplicative_moments_ver2}
m_n(p_d \boxtimes_d q_d ) 
	\ \ 	= 
			\sum_{\pi\in NC(n)} 
				m_\pi(p_d) \,
				\kappa_{Kr(\pi)}(q_d) 
		\	- \
			\frac{n}{2d}
			\sum_{\substack{ t+s = n \\ \sigma \in S_{NC}(t,s) } } 
				\frac{m_{\sigma}(p_d) \, \kappa_{Kr_{t,s}(\sigma)}(q_d)}{ts} 
		\ \	+ \ \
			o\left(1/d\right) .
\end{equation}
Moreover, letting $q_d = (x-1)^d$, we have that $p_d \boxtimes_d q_d = p_d$ for any $p_d  \in \pols_d(\rr)$ and  $m_n(q_d)=1$ for $1 \leq n \leq d$, 
so \eqref{eqn:truncated_expasion_multiplicative_moments} yield the following. 
\begin{corollary}[Terms of order $\Theta(1/d)$, {\cite[Theorem 1.3]{arizmendi2023finite}}] \label{thm:moment_cumulant_approx_2}
For any $p_d \in \pols_d(\rr) $ and  $1 \leq n \leq d$, we have that 
\begin{equation}\label{eqn:moment_cumulant_approx_2}
m_n(p_d)  
\ \ \, =  
		\sum_{\pi\in NC(n)} \kappa_\pi(p_d)  
	\ -  \ 
		\frac{n}{2 d} 
		\sum_{ \substack{t,s=n\\  \sigma \in S_{NC}(t,s)}} 
		\frac{\kappa_\sigma(p_d)}{ts} 
	\ \ + \ \ 
		o(d^{-1})
\end{equation}
\end{corollary}
From \eqref{eqn:moment_cumulant_approx_2}, it readily follows that convergence in moments is equivalent to convergence in cumulants in finite-free probability. 
\begin{proposition}[{\cite[Theorem 5.5]{arizmendi2018cumulants}}]\label{prop:convergence_cumulants}
Let $\mfp=(p_d)_{d\geq 1}$ be a sequence of polynomials with $p_d \in \pols_d(\rr) $ for every $d \geq 1$ and let $\mu\in \MM(\R)$ be a probability distribution with finite moments $(m_n(\mu))_{n \geq 1}$ of all orders. 
Then $\mfp=(p_d)_{d\geq 1}$ has limiting distribution $\mu$, in the sense that $ m_n(p_d) \to m_n(\mu)$ as $d \to \infty$ for all $n \geq 0$, if and only if 
the finite-free cumulants of $\mfp=(p_d)_{d\geq 1}$ converge to the free cumulants of $\mu$, i.e.,  $\kappa_n(p_d) \to r_n(\mu)$ as $d \to \infty$ for all $n \geq 0$.
\end{proposition}
Furthermore, since cumulants $\kappa_n(\cdot)$ and $r_n(\cdot)$ linearize additive convolutions $\boxplus_d$ and $\boxplus$, respectively, we obtain that $\boxplus_d$ converges to $\boxplus$ as $d \to \infty$. 

\begin{proposition}[{\cite[Corollary 5.5]{arizmendi2018cumulants}}] \label{prop:finiteAsymptotics_additive}
Suppose $\mfp=(p_d)_{d\geq1}$ and $\mfq=(q_d)_{d\geq1}$ are sequences of polynomials with limiting distributions $\mu, \nu \in \MM(\R)$, respectively. 
Then the finite-free additive convolution $\mfp \boxplus \mfq := (p_d\boxplus_d q_d)_{d\geq1}$ has limiting distribution $\mu \boxplus \nu$. 
\end{proposition}
It follows from \eqref{eqn:truncated_expasion_multiplicative_cumulants}  that $\boxtimes_d$ converges to $\boxtimes$. 
\begin{proposition}[{\cite[Theorem 1.4]{arizmendi2023finite}}]
    \label{prop:finiteAsymptotics_multiplicative}
Suppose $\mfp=(p_d)_{d\geq1}$ and $\mfq=(q_d)_{d\geq1}$ are sequences of polynomials with limiting distributions $\mu \in \MM(\R)$ and $\nu \in  \MM(\R_{\geq 0})$, respectively. 
Then the finite-free multiplicative convolution $\mfp \boxtimes \mfq := (p_d\boxtimes_d q_d)_{d\geq1}$ has limiting distribution $\mu \boxtimes \nu$. 
\end{proposition}


\section{Combinatorial relations for infinitesimal moments and finite free cumulants}
\label{sec:combinatorial.description}

In this section, we calculate the infinitesimal moments arising from the finite-free additive and multiplicative convolution of polynomials with given infinitesimal distribution.

\subsection{Finite Cumulant Fluctuations}

To start, we show that a sequence of polynomials with limiting root distribution has infinitesimal moments of all orders if and only if all finite-free cumulants have an expansion in $1/d$ up to a correction terms of order $o(1/d)$. 
\begin{lemma}\label{lem:infinitesimal_single_polynomial}
For  a sequence of polynomials $\mfp = (p_d)_{d\geq1}$ with limiting distribution $\mu \in \MM(\R)$,  
the following two conditions are equivalent:   

\begin{enumerate}[(1)]
\item The sequence $\mfp=(p_d)_{d\geq1}$ has infinitesimal moments of all orders, 
namely, there exists a sequence  $(m_n'(\mfp))_{n\geq 1}$ that does not depend on $d$ such that 
\begin{equation}\label{eqn:hypothesis_cumulants_ver_2} 
m_n(p_d)=m_n(\mu)+ \frac{1}{d}  \, m_n' (\mfp) + o(1/d) 
\quad \forall \ n \geq 1.
\end{equation}

\item There exists a sequence $(\widehat{r}_n(\mfp))_{n\geq 1}$ that does not depend on $d$ such that 
\begin{equation}\label{eqn:hypothesis_cumulants_ver_1}
\kappa_n (p_d)
= r_n (\mu)+ \frac{1}{d} \, \widehat{r}_n (\mfp) + o(1/d) 
\quad \forall \ d \geq n \geq 1.
\end{equation}
\end{enumerate}
Moreover, if any of two conditions above holds, then the sequences $(m_n'(\mfp))_{n\geq 1}$ and $(\widehat{r}_n(\mfp))_{n\geq 1}$ determine each other through the relation 
\begin{equation}\label{eqn:mom.cum.inf}
m'_n (\mfp) \ \ \ =\sum_{\pi \in NC(n)} 
            \sum_{V \in \pi } 
                    \widehat{r}_{\abs{V}} (\mfp) \cdot 
                    r_{\pi \setminus V}( \mu )
\ \  -  \ \ 
	\frac{n}{2}  \sum_{\substack{ t+s=n\\ \sigma \in S_{NC}(t,s)  } }
					\frac{  r_{\sigma} (\mu) } {ts}   
\quad \forall \ n \geq 1.
\end{equation}
\end{lemma}

\begin{proof}
$\textit{(2)} \Rightarrow \textit{(1)}$.
For any integer $j\geq 1$, we have that  $\kappa^{}_j(p_d) = r_j(\mu) + o(1)$ since $\kappa^{}_j(p_d)  \to r_j(\mu)$ as $d \to \infty$, due to Proposition \ref{prop:convergence_cumulants}. 
Combing this with \eqref{eqn:hypothesis_cumulants_ver_1}, and applying both of them to the relation \eqref{eqn:moment_cumulant_approx_2}, we obtain that 
\begin{align*}
m_n(p_d) 
\,   &= \sum_{\pi \in NC(n)}\prod_{V\in\pi} \left( r_{|V|} (\mu)+ \frac{1}{d} \,  \widehat{r}_{|V|}(\mfp) + o(1/d) \right) 
\,\, - \,\, 
			\frac{n}{2d}  \sum_{\substack{ t+s=n\\ \sigma \in S_{NC}(t,s)  } }
					\frac{ 1 } {ts} \, \prod_{V \in \sigma} \Big( r_{\abs{V}} (\mu) + o(1) \Big)
\,\, + \,\, 
o(1/d)\\
&= \sum_{\pi \in NC(n)} r_\pi(\mu)  
\,\,+ \,\,
     \frac{1}{d} 
            \sum_{\pi \in NC(n)} 
            \sum_{V \in \pi } 
                    \widehat{r}_{\abs{V}}(\mfp) \cdot 
                    r_{\pi \setminus V}( \mu )
\,\, - \,\,
	\frac{n}{2d}  \sum_{\substack{ t+s=n\\ \sigma \in S_{NC}(t,s)  } }
					\frac{  r_{\sigma} (\mu) } {ts} 
\,\, + \,\,
    o(1/d). 
\end{align*}
However, due to the free moment-cumulant formula \eqref{eqn:moment-cumulant-combinatorial_intro}, the last line yields precisely  \eqref{eqn:hypothesis_cumulants_ver_2} with each $m'_n(\mfp)$ given by \eqref{eqn:mom.cum.inf}.  \\ 

$\textit{(1)} \Rightarrow \textit{(2)}$. 
We will prove that \eqref{eqn:hypothesis_cumulants_ver_1} holds by induction on $n$. 
For $n=1$, we have that $m_1(\mu) = r_1(\mu)$, so  \eqref{eqn:moment_cumulant_approx_2} and  \eqref{eqn:hypothesis_cumulants_ver_2} yield
\begin{equation*}
\kappa^{}_1(p_d) = m_1(p_d) + o(1/d^2) = r_1(\mu) + \frac{1}{d} \, m'_1 + o(1/d) . 
\end{equation*}
Hence, we can take $\widehat{r}_1 (\mfp) = m'_1 (\mfp)$. 
Assume now that there exist $\widehat{r}_1(\mfp) , \widehat{r}_2(\mfp), \ldots, \widehat{r}_{n-1}(\mfp)$ such that 
\begin{equation*}
\kappa_j (p_d)
= r_j (\mu)+ \frac{1}{d} \, \widehat{r}_j (\mfp) + o(1/d) \quad \text{for} \quad j=1,2,\dots,n-1.
\end{equation*}
Given partition $\pi \in NC(n)$, the product of finite-free cumulants $\kappa_\pi(p_d)=\prod_{V \in \pi} \kappa^{}_{\abs{V}}(p_d)$ contains the term $\kappa_n(p_d)$ as a factor only if $\pi = 1_n$. 
So, we can re-write \eqref{eqn:moment_cumulant_approx_2} as 
\begin{equation*}
\kappa^{}_n(p_d) \, =
			 \, m_n(p_d)
			\, - \,  
					\sum_{\substack{\pi \in NC(n)\\ \pi\neq 1_n}}\kappa^{}_\pi(p_d) 
			\, + \, 
			\frac{n}{2d}  \sum_{\substack{ t+s=n\\ \sigma \in S_{NC}(t,s) } }
					\frac{\kappa^{}_\sigma(p_d) } {ts}
	\,\,  + \,\, o(1/d) \, .
\end{equation*}
Moreover, our induction hypothesis implies that  
\begin{equation*}
\kappa_\pi(p_d) 
=
\prod_{V \in \pi} \left(   r_{\abs{V}} (\mu)+ \frac{1}{d} \, \widehat{r}_{\abs{V}}(\mfp) + o(1/d) \right)
=
 r_\pi(\mu) 
			\,\, + \,\,
    		 \frac{1}{d} 
            \sum_{V \in \pi } 
                    \widehat{r}_{\abs{V}}(\mfp) \cdot 
                    r_{\pi \setminus V}( \mu )
			\,\, +\,\, 
			o(1/d) 
\end{equation*}
whenever $1_n \neq \pi \in NC(n)$; additionally, since $\kappa^{}_j(p_d)  \to r_j(\mu)$ as $d \to \infty$  for any integer $j\geq 1$, meaning that  $\kappa^{}_j(p_d) = r_j(\mu) + o(1)$,  we obtain that 
\begin{equation*}
\frac{1}{d} \, \kappa^{}_\sigma(p_d) 
=
	\frac{1}{d} \, \prod_{V \in \sigma} \Big(  r_{\abs{V}}(\mu) + o(1)  \Big)
=
	\frac{1}{d} \, r^{}_\sigma(\mu) \,\, +\,\, o(1/d) 
\end{equation*}
for every permutation $\sigma \in S_{NC}(t,s)$ with $t+s=n$. 
Thus, we have that  
\begin{equation*}
\kappa_n(p_d)  \, =
			\  \, m_n(p_d)
			\, - \,  
					\sum_{\substack{\pi \in NC(n)\\ \pi\neq 1_n}} r_\pi(\mu) 
			\,\, - \,\,
    		 \frac{1}{d} 
            \sum_{\substack{\pi \in NC(n)\\ \pi\neq 1_n}} 
            \sum_{V \in \pi } 
                    \widehat{r}_{\abs{V}}(\mfp) \cdot 
                    r_{\pi \setminus V}( \mu )
			\,\,  + \,\, 
					\frac{n}{2d}  \sum_{\substack{ t+s=n\\ \sigma \in S_{NC}(t,s) } } 
						\frac{ r_\sigma( \mu ) } {ts}
			\,\, +\,\, 
			o(1/d) .
\end{equation*}
But, the free moment-cumulant formula \eqref{eqn:moment-cumulant-combinatorial_intro} implies that 
\begin{equation*}  
r_n(\mu) -m_n(\mu)	\ \ = -  
 \sum_{\substack{\pi \in NC(n)\\ \pi\neq 1_n}} r_\pi(\mu) \, ,
\end{equation*}
and we know that $m_n(p_d)=m_n(\mu)+ \frac{1}{d}  \, m_n'(\mfp)+o(1/d)$, so we obtain  
\begin{equation*}
\kappa^{}_n(p_d)  =  r_n(\mu) 
			\,\, + \,\,
    		 \frac{1}{d} 
\left( m'_n(\mfp) -
            \sum_{\substack{\pi \in NC(n)\\ \pi\neq 1_n}} 
            \sum_{V \in \pi } 
                    \widehat{r}_{\abs{V}}(\mfp)\cdot 
                    r_{\pi \setminus V}( \mu )
			\,\,  + \,\, 
				\frac{n}{2}  \sum_{\substack{ t+s=n\\ \sigma \in S_{NC}(t,s) } } 
						\frac{ r_\sigma( \mu ) } {ts}
\right)
			\,\,  + \,\, 
			o(1/d) \, . 
\end{equation*}
We can then take $\widehat{r}_n$ as the expression inside the parentheses.
\end{proof}

\begin{definition} \label{def:fluctiation.cumulants}
For a sequence of polynomials $\mfp=(p_d)_{d\geq1}$ with infinitesimal distribution $(\mu,\mu')$,
the elements of the sequence $(\widehat{r}_n(\mfp))_{n\geq 1}$  from Lemma \ref{lem:infinitesimal_single_polynomial} are called the \emph{(finite-free) cumulant fluctuations of} $\mfp$. 
\end{definition}
By Proposition \ref{prop:infinitesimal-moment-cumulant-functional}, the relation \eqref{eqn:mom.cum.inf} can be inverted to solve for the cumulant fluctuations.
\begin{corollary}\label{cor:cum_mom_inf_h}
In the context of Lemma \ref{lem:infinitesimal_single_polynomial}, letting $
h_n (\mu) =
 \frac{n}{2} \sum_{ \substack{s+t=n\\ \sigma \in S_{NC}(s,t)}} \frac{r_\sigma(\mu)}{st}$, the relation \eqref{eqn:mom.cum.inf} is equivalent to  
\begin{equation}\label{eqn:cum.mom.inf}
\widehat{r}_n (\mfp)
    =
    \sum_{\pi \in NC(n)}  
            \text{Möb}(\pi,1_n) 
            \sum_{V \in \pi } 
                    (m'_{\abs{V}} (\mfp) +h_{\abs{V}} (\mu) )  
					\cdot 
                    m_{\pi \setminus V}( \mu ) 
\quad \forall \ n \geq 1.
\end{equation}
\end{corollary}

\subsection{Infinitesimal moments for additive convolution}

The previous lemma allows us to compute directly the infinitesimal moments for finite-free additive convolution. 

\begin{theorem}\label{thm:infinitesimal_additive_finite_free cumulants}
Suppose $\mfp=(p_d)_{d\geq1}$ and $\mfq=(q_d)_{d\geq1}$ are  two sequences of polynomials with infinitesimal distributions $(\mu,\mu')$ and $(\nu,\nu')$, respectively. 
Then, the finite-free additive convolution $\mfp \boxplus \mfq = (p_d\boxplus_d q_d)_{d\geq1}$ has  infinitesimal distribution $(\rho,\rho')$ where $\rho=\mu\boxplus\nu$ and the infinitesimal moments $m_n(\rho')$ are given by
\begin{equation}\label{eqn:mom.cum.inf.additive}
m_n(\rho') \ \ = 
		\sum_{\pi\in NC(n)} 
		\sum_{V \in\pi} 
				\left[ \widehat{r}_{\abs{V}}(\mfp) + \widehat{r}_{\abs{V}}(\mfq)\right]
					 r_{\pi\setminus V}(\mu \boxplus \nu) 
		\ \ - \ \ 
		\frac{n}{2} 
		\sum_{ \substack{s+t=n\\ \sigma \in S_{NC}(s,t)}} 	\frac{r_\sigma(\mu\boxplus\nu)}{st} 
\qquad \forall \ n \geq 1.
\end{equation}

\end{theorem}
\begin{proof}
%
Lemma \ref {lem:infinitesimal_single_polynomial} implies that there exist sequences $(\widehat{r}_n(\mfp))_{n\geq 1}$ and $(\widehat{r}_n(\mfq))_{n\geq 1}$ such that 
\begin{equation*}
\kappa_n(p_d)= r_n(\mu)+ \frac{1}{d}  \, \widehat{r}_n (\mfp)+o(1/d) \quad \text{and} \quad 
\kappa_n(q_d)= r_n(\nu)+ \frac{1}{d}  \, \widehat{r}_n (\mfq)+o(1/d) 
\qquad \forall \ n \geq 1.
\end{equation*}
Moreover, the free cumulants $\kappa_n(\, \cdot \,)$ and $r_n(\, \cdot \,)$ linearize the additive convolutions $p_d \boxplus_d q_d$ and $\mu \boxplus \nu$, respectively, and hence adding $\kappa_n(p_d)$ and $\kappa_n(p_d)$ yields
\begin{equation*}
\kappa_n(p_d \boxplus_d q_d ) 
= r_n(\mu \boxplus \nu)+ \frac{1}{d}  \, \Big( \widehat{r}_n (\mfp) + \widehat{r}_n (\mfq) \Big) + o(1/d) 
\qquad \forall \ n \geq 1.
\end{equation*}
Now, the sequence $\mfp \boxplus \mfq = (p_d\boxplus_d q_d)_{d\geq1}$ has limiting distribution $\mu \boxplus \nu$ since finite-free additive convolution $\boxplus_d$ converges to free additive convolution $\boxplus$. 
Therefore, a direct application of Lemma \ref {lem:infinitesimal_single_polynomial} yields that $\mfp \boxplus \mfq$ has infinitesimal moments given by \eqref{eqn:mom.cum.inf.additive}. 
\end{proof}

\subsection{Cumulant fluctuations for  multiplicative convolution}

For the finite-free multiplicative convolution, if we show that $\mfp \boxtimes \mfq = (p_d\boxtimes_d q_d)_{d\geq1}$ has cumulant fluctuations whenever  $\mfp=(p_d)_{d\geq1}$ and $\mfq=(q_d)_{d\geq1}$ have them, 
then Lemma \ref{lem:infinitesimal_single_polynomial} ---and the fact that finite-free multiplicative convolution $\boxtimes_d$ converges to free multiplicative convolution $\boxtimes$--- implies that $\mfp \boxtimes \mfq$ has infinitesimal moments given by 
\begin{equation}\label{eqn:infinitesimal_moments_ffc_product}
m'_n(\mfp \boxtimes \mfq)  
		\ \ = 
				\sum_{\pi\in NC(n)} 
				\sum_{V \in\pi} 
					\widehat{r}_{\abs{V}} ( \mfp \boxtimes \mfq ) 
					\cdot
					r_{\pi\setminus V}( \mu \boxtimes \nu ) 
	\quad -    
				\sum_{ \substack{t+s=n\\ \sigma \in S_{NC}(t,s)}} 				\frac{r_\sigma(\mu\boxtimes\nu)}{st}.
\end{equation}
This is the content of our next theorem, which, in particular, provides a formula for $\widehat{r}_{n} ( \mfp \boxtimes \mfq )$. 
\begin{theorem}\label{thm:infinitesimal_multiplicative_finite_free}
Suppose $\mfp=(p_d)_{d\geq1}$ and $\mfq=(q_d)_{d\geq1}$ are two sequences of polynomials with infinitesimal distributions   $(\mu,\mu')$ and $(\nu,\nu')$, respectively. 
Then, the finite-free cumulants $\kappa_n (p_d \boxtimes_d q_d)$ satisfy  the relation  
 \begin{equation}
\kappa_n (p_d \boxtimes_d q_d)
=r_n(\mu\boxtimes \nu)+ \frac{1}{d} \, \widehat r_n (\mfp \boxtimes \mfq)+o(1/d) 
\quad \forall \ d \geq n \geq 1
\end{equation}
where
\begin{align}\label{eqn:rhats_multiplicative}
\widehat r_n(\mfp \boxtimes \mfq) 
	\ \ =  & 
			\sum_{\pi\in NC(n)}\left(
					\sum_{V\in\pi} 
						\widehat{r}_{|V|}(\mfp) r_{\pi\backslash V}(\mu)
						\cdot 
						r_{Kr(\pi)}(\nu) 
		+
			\sum_{W\in Kr(\pi)}
					r_\pi(\mu)  
					\cdot
					\widehat{r}_{|W|}(\mfq) r_{Kr(\pi)\backslash W}(\nu)
			\right)
\nonumber \\ &  -
			\frac{n}{2} 
			\sum_{\substack{ t+s = n \\ \sigma \in S_{NC}(t,s) } } 
				\frac{r_{\sigma}(\mu) \, r_{Kr_{t,s}(\sigma)}(\nu)}{ts} 
\end{align}
Consequently, the finite-free multiplicative convolution $\mfp \boxtimes \mfq = (p_d\boxtimes_d q_d)_{d\geq1}$ has infinitesimal moments given by \eqref{eqn:infinitesimal_moments_ffc_product}. 
\end{theorem}
\begin{proof}
By Lemma \ref{lem:infinitesimal_single_polynomial}, our hypothesis are equivalent to the existence of probability distributions $\mu$ and $\nu$ as well as sequences $(\widehat{r}_n(\mfp))_{n\geq 1}$ and $(\widehat{r}_n(\mfq))_{n\geq 1}$ such that 
\begin{equation*}
\kappa_n(p_d)= r_n(\mu)+ \frac{1}{d}  \, \widehat{r}_n (\mfp)+o(1/d) \quad \text{and} \quad 
\kappa_n(q_d)= r_n(\nu)+ \frac{1}{d}  \, \widehat{r}_n (\mfq)+o(1/d) 
\quad \forall \ d \geq n \geq 1. 
\end{equation*} 
Now, we know from \eqref{eq.cumulant.of.products} from Corollary \ref{cor:truncated_expasions} that
\begin{equation*}
\kappa_n (p_d\boxtimes_d q_d)
 \ \	= 
			\sum_{\pi\in NC(n)} 
				\kappa_\pi(p_d) \,
				\kappa_{Kr(\pi)}(q_d) 
			-
			\frac{n}{2d}
			\sum_{\substack{ t+s = n \\ \sigma \in S_{NC}(t,s) } } 
				\frac{\kappa_{\sigma}(p_d) \, \kappa_{Kr_{t,s}(\sigma)}(q_d)}{ts} 
			+
			o\left(1/d\right)
\end{equation*}
The conclusion follows from expanding each product $\kappa_\pi(p_d) \kappa_{Kr(\pi)}(q_d) $ as  
\begin{align*}
& \kappa_\pi(p_d) \, \kappa_{Kr(\pi)}(q_d)
		=
			\prod_{V \in \pi} 
				\left( r_{\abs{V}}(\mu)+\frac{1}{d}\, \widehat{r}_{\abs{V}}(\mfp) \right)
            \cdot
			\prod_{W \in Kr(\pi)} 
				\left( r_{\abs{W}}(\nu)+\frac{1}{d}\, \widehat{r}_{\abs{W}}(\mfq) \right)
			\ \ + \ \
			o(1/d)
\\
&=
		r_\pi(\mu) \, r_{Kr(\pi)}(\nu)
    \ \ + \ \ 
		\frac{1}{d}
		\sum_{V\in\pi} 
			\widehat{r}_{|V|}(\mfp)r_{\pi\backslash V}(\mu)
			\cdot 
			r_{Kr(\pi)}(\nu)   
	\ \ + \ \  
		\frac{1}{d}
		\sum_{W\in Kr(\pi)}
			r_\pi(\mu)
			\cdot
			\widehat{r}_{|W|}(\mfq) 
			r_{Kr(\pi) \backslash W}(\nu)   
	\ \ + \ \  
			o(1/d),
\end{align*}
and similarly for each product $\kappa_{\sigma}(p_d) \, \kappa_{Kr_{t,s}(\sigma)}(q_d)$, and the fact that $r_n(\mu\boxtimes \nu) = \sum_{\pi\in NC(n)} r_\pi(\mu) \, r_{Kr(\pi)}(\nu)$. 
\end{proof}
Using \eqref{eq:multiplicativemomentcumulant} instead of \eqref{eq.cumulant.of.products} and considering the relations
\begin{equation*}
m_n(p_d)= m_n(\mu)+ \frac{1}{d}  \, m'_n (\mfp)+o(1/d) \quad \text{and} \quad 
\kappa_n(q_d)= r_n(\nu)+ \frac{1}{d}  \, \widehat{r}_n (\mfq)+o(1/d) , 
\end{equation*}
similar arguments as in the previous proof yield another derivation for  $m'_n(\mfp \boxtimes \mfq)$. 
\begin{theorem}\label{prop:moments_for_multiplicative}
Suppose $\mfp=(p_d)_{d\geq1}$ and $\mfq=(q_d)_{d\geq1}$ are two sequences of polynomials with infinitesimal distributions   $(\mu,\mu')$ and $(\nu,\nu')$, respectively. 
Then, the moments $m_n (p_d \boxtimes_d q_d)$ satisfy  the relation  
 \begin{equation}
m_n (p_d \boxtimes_d q_d)
=m_n(\mu\boxtimes \nu)+ \frac{1}{d} \,  m'_n (\mfp \boxtimes \mfq)+o(1/d) 
\quad \forall \ n \geq 1
\end{equation}
where
\begin{align}\label{eqn:mprime_multiplicative}
m'_n(\mfp\boxtimes \mfq)
	\ \ = & 
			\sum_{\pi\in NC(n)}\left(
				\sum_{V\in\pi} 
					m'_{|V|}(\mfp) m_{\pi\backslash V}(\mu)
					\cdot 
					r_{Kr(\pi)}(\nu) 
				\ \ +
				\sum_{W\in Kr(\pi)}
					m_\pi(\mu)
					\cdot 
					\widehat{r}_{|W|}(\mfq) r_{Kr(\pi)\backslash W}(\nu) 
			\right) 
\nonumber
\\ & -
			\frac{n}{2}
			\sum_{\substack{ t+s = n \\ \sigma \in S_{NC}(t,s) } } 
				\frac{m_{\sigma}(\mu) \, r_{Kr_{t,s}(\sigma)}(\nu)}{ts}.
\end{align}
\end{theorem}


\section{Infinitesimal Cauchy transform in finite free probability}
\label{Functional}

In this section, we translate our results from Section \ref{sec:combinatorial.description} into functional relations of formal power series. 
For this, given a sequence of polynomials $\mfp=(p_d)_{d\geq1}$ with infinitesimal distribution $(\mu,\mu')$, besides the Cauchy transform and the $R$-transform \eqref{eqn:cauchy_and_r_transform}, and their infinitesimal counterparts \eqref{eqn:infinitesimal_cauchy_and_r}, we need to consider two additional formal power series: 
\begin{equation*}
\widehat{R}_{\mfp}(z) \, =  \sum_{n=1}^{\infty} \widehat{r}_{n}(\mfp) \, z^{n-1}  
\end{equation*}
determined by the cumulant fluctuations $(\widehat{r}_{n}(\mfp))_{n \geq 1}$ from Lemma \ref{lem:infinitesimal_single_polynomial} and Corollary \ref{cor:cum_mom_inf_h}; and 
\begin{equation*}
\gdos_{\mu}(z) \, = \sum_{n=2}^{\infty} h_n (\mu) \, z^{-n-1}  
\quad \text{with} \quad 
h_n (\mu) \ = \, \frac{n}{2} \sum_{\substack{t+s=n \\ \sigma \in S_{NC}(t,s)}} \frac{r_\sigma(\mu)} { ts } 
\end{equation*}
referred to as the \emph{$H$-transform} of $\mu$. 
The latter can also be expressed in terms of the Cauchy transform $G_\mu$ as follows. 

\begin{proposition} \label{prop:HG-functional}
Let $(m_n)_{n \geq 1}$ and $(r_n)_{n \geq 1}$ be any two sequences of numbers. 
Suppose that $G(z)  = z^{-1} + \sum_{n=1}^{\infty}  m_n \, z^{-n-1} $ and $K(z) = z^{-1} + \sum_{n=1}^{\infty} r_n \, z^{n-1} $ satisfy any of the four equivalent conditions from Proposition \ref{prop:basic.funct.relation}. 
Then, for any formal power series $\gdos(z) = \sum_{n=2}^{\infty}  h_n \, z^{-n-1}$, we have that 
\begin{equation}\label{eqn:HG-functional}
H = \frac{G'}{G}-\frac{G''}{2G'}
\qquad \text{if and only if} \qquad 
h_n \ = \, \frac{n}{2} \sum_{\substack{t+s=n \\ \pi \in S_{NC}(t,s)}} \frac{r_\pi} { ts } 
\quad \forall n \geq 2.
\end{equation}
\end{proposition}
\begin{proof}
This is an implicit consequence of \cite[s (5.4) and (5.6)]{arizmendi2023finite}
\end{proof}

When doing computations, and in the context of the previous proposition, it comes handy to consider the reciprocal formal power series $F(z)=1 / G(z) $, taken as a series in $z$ of the form $z+f_2z^2+f_3z^3+\cdots$ with formal derivative  $F'(z)=1+\tilde f_1z+\tilde f_2z^2+\cdots$. 
One can then show that the left-hand side of \eqref{eqn:HG-functional} can be expressed as 
\begin{equation} \label{eqn:logF}
H=-\frac{1}{2}\left( \ln(F') \right)'= \frac{-F''}{2F'}.
\end{equation}
The following lemma describes the relation between two $H$-transforms when the corresponding $G$ are related through composition. 
This relation will be used later on. 

\begin{lemma}
\label{lem:H.circ.function}
Let $G_1$ and $G_2$ be formal power series of the form $z^{-1}+a_1z^{-2}+\cdots$, and $J$ a formal power series of the form $1+ b_1z+b_2 z^2+\cdots$.
Put $H_i = \frac{G'_i}{G_i} - \frac{G''_i}{2G'_i} $ for $i=1,2$.  
If $G_1\circ J = G_2$, then
\begin{equation}
\label{eq:H.circ.function}
(H_1\circ J)J'=H_2+\frac{J''}{2J'}.
\end{equation}
\end{lemma}
\begin{proof}
Let $\mathcal{H}_i = - \frac{1}{2} \ln \left( G'_i / (G_i)^2 \right)$, for $i=1,2$, and 
note that $\mathcal{H}'_i = H_i$. 
Thus, in particular, we have that 
$
(\mathcal{H}_1 \circ J)' =  (H_1\circ J) \cdot J'
$.  
Now, since $G_2 = G_1 \circ J$, and hence $ G'_2 = (G'_1\circ J) \cdot J'$, we also get that 
\begin{equation*}
\mathcal{H}_1 \circ J 
= 
		-\frac{1}{2} \ln \left(  \frac{(G'_1 \circ J) \cdot J' }{(G_1 \circ J)^2} \cdot \frac{1}{J'}\right) 
=
	  \mathcal{H}_2  + \frac{1}{2}\ln \left(J'\right) 
\end{equation*}
Differentiating both sides of the last equation, we obtain the desired relation. 
\end{proof}


\subsection{Cauchy transform and finite-free cumulants}

We first establish various functional relations in the case of a single sequence of polynomials $\mfp=(p_d)_{d\geq1}$ with infinitesimal distribution $(\mu,\mu')$,  
beginning with the functional equation that expresses the infinitesimal Cauchy transform $G_{\mu'}$ in terms of cumulant fluctuation transform $\widehat{R}_\mfp$, the Cauchy transform $G_\mu$, and the $\gdos$-transform $\gdos_\mu$.

\begin{lemma}\label{lemma:additive.delta.zero}
For any sequence of polynomials $\mfp=(p_d)_{d\geq1}$ with infinitesimal distribution $(\mu,\mu')$, the following relation holds 
\begin{equation} \label{eq. qis0}
G_{\mu'} = - (\widehat{R}_{\mfp} \circ G_{\mu} ) \cdot G'_{\mu} - \gdos_{\mu} .
\end{equation}
\end{lemma}

\begin{proof}
From Lemma \ref{lem:infinitesimal_single_polynomial}, we have that 
\begin{equation*}
m_n(p_d)=m_n(\mu)+ \frac{1}{d}  \, m_n'(\mfp)+o(1/d)
\end{equation*}
with each $m_n'(\mfp)$ given by 
\begin{equation*}
m'_n(\mfp) \ \ \ =\sum_{\pi \in NC(n)} 
            \sum_{B \in \pi } 
                    \widehat{r}_{\abs{B}} (\mfp) \cdot 
                    r_{\pi \setminus B}( \mu )
\quad - \ 
            \sum_{ \substack{ t+s = n \\ \sigma \in S_{NC}(t,s) }} 
                r_{\sigma } ( \mu )  \, .
\end{equation*}
But, from Proposition  \ref{prop:infinitesimal-moment-cumulant-functional} and Proposition \ref{prop:HG-functional}, we know that 
\begin{equation*}
- (\widehat{R}_{\mfp} \circ G_{\mu} (z)) \cdot G'_{\mu} (z) =\sum_{n=1}^{\infty} z^{-n-1} \sum_{\pi \in NC(n)} 
            \sum_{B \in \pi } 
                    \widehat{r}_{\abs{B}} (\mfp) \cdot 
                    r_{\pi \setminus B}( \mu)  \quad
\text{and} 
\quad
H_\mu(z )=\sum_{n=1}^{\infty} z^{-n-1} \sum_{ \substack{ t+s = n \\ \sigma \in S_{NC}(t,s) }} 
                r_{\sigma } ( \mu )  \, .
\end{equation*}
Consequently, we obtain that $\displaystyle G_{\mu'} =\sum_{n=1}^{\infty}  m'_n  z^{-n-1}= - (\widehat{R}_{\mfp} \circ G_{\mu} ) \cdot G'_{\mu} - H_{\mu}$.
\end{proof}

From Lemma \ref{lem:infinitesimal_single_polynomial} and Corollary \ref{cor:cum_mom_inf_h}, we know that the $1/d$ behavior of  a sequence of polynomials with limiting distribution can be retrieved by either cumulant fluctuations or  infinitesimal moments, with \eqref{eqn:mom.cum.inf} and \eqref{eqn:cum.mom.inf} describing how to go from one to the other explicitly. 
Our next lemma establishes this in terms of the related formal power series. 
\begin{lemma}
\label{lem:infinitesimal_single_polynomial_2}
For a sequence of polynomials $\mfp=(p_d)_{d\geq1}$ with infinitesimal distribution $(\mu,\mu')$, the following two functional relations are equivalent:
\begin{equation}
\label{eqs:GandR}
G_{\mu'}+\gdos_\mu = - (\widehat{R}_\mfp \circ G_{\mu}) \cdot G'_{\mu} %
\qquad \text{ and }\qquad 
\widehat{R}_{\mfp}= -\left(G_{\mu'}\circ K_\mu+\gdos_\mu \circ K_\mu\right) \cdot K'_\mu.
\end{equation}
As an immediate consequence, we obtain from \eqref{eqn:Mingo.relation} and \eqref{eqs:GandR} that $H_\mu=0$ if and only if $\widehat{R}_\mfp=\rinf_{\mu,\mu'}$.  
\end{lemma}

\begin{proof}
This follows directly from Proposition \ref{prop:infinitesimal-moment-cumulant-functional}. 
Indeed, letting $G=G_\mu$, $K=K_\mu$, $r=\widehat{R}_\mfp$ and $g=G_{\mu'}+H_\mu$, we have that  $G(K(z))=z$, and hence Proposition \ref{prop:infinitesimal-moment-cumulant-functional} gives  the equivalence between the relations 
\begin{equation*}
g(z) = -r(G(z)) \cdot G'(z) \qquad \text{ and }\qquad r(z) = -g(K(z)) \cdot K'(z) \, ,
\end{equation*}
which are exactly the desired formulas \eqref{eqs:GandR}. 
\end{proof}

Before applying the previous lemma to finite-free convolution, let us conclude this part by providing the relation between the transform $\widehat{R}_{\mfp}(z)$ and the infinitesimal $R$-transform  $\rinf_{\mu,\mu'}(z)$.
\begin{lemma}
\label{lem:infinitesimal_single_polynomial_3}
Under the conditions of Lemma \ref{lem:infinitesimal_single_polynomial_2}, the infinitesimal $R$-transform satisfies
\begin{equation}
\label{eqs:RhatRinf} 
\widehat{R}^{}_\mfp(z)
=
        \rinf_{\mu,\mu'} (z)-(H_\mu \circ K_\mu) \cdot  K'_\mu
=
        \rinf_{\mu,\mu'} (z)- \frac{K''_\mu(z)}{2K'_\mu(z)}-\frac{1}{z}.
\end{equation}
\end{lemma}

\begin{proof}
From Lemma \ref{lem:infinitesimal_single_polynomial_2}, we know that $\widehat{R}_{\mfp}= -\left(G_{\mu'}\circ K_\mu \right) \cdot K'_\mu -\left(\gdos_\mu \circ K_\mu\right) \cdot K'_\mu $. 
We readily identify the term $-\left(G_{\mu'}\circ K_\mu \right) \cdot K'_\mu$ as the infinitesimal $R$-transform $\rinf_{\mu,\mu'}$ corresponding to the pair $(\mu,\mu')$, see \eqref{eqn:Mingo.relation}.  
To deal with the term $-\left(\gdos_\mu \circ K_\mu\right) \cdot K'_\mu$, we simply apply Lemma \ref{lem:H.circ.function} (in a modified version, composing with $1/z$).
First, we take $G_1(z)=G_\mu(z)$, $J(z)=K_\mu(z)$, and $G_2(z)=z$.
Then, we note that $G_2 = G_1 \circ J$ since $G_\mu(K_\mu(z))=z$. 
Thus, Lemma \ref{lem:H.circ.function} implies 
\begin{equation*}
-(H_1 \circ K_\mu ) \cdot K'_\mu = - H_2 - \frac{K''_\mu}{2K'_\mu}
\text{\qquad with \qquad }
H_i = \frac{G'_i}{G_i} - \frac{G''_i}{2G'_i}.
\end{equation*}
But, we have that $H_1(z)=\gdos_\mu(z)$ and $H_2(z)
= \tfrac{1}{z}$. 
So, we obtain the desired conclusion. 
\end{proof}

\subsection{Additive convolution}

\begin{theorem}\label{thm:infinitesimal_additive_finite_free}
Suppose that $\mfp=(p_d)_{d\geq1}$ and $\mfq=(q_d)_{d\geq1}$ are two sequences of polynomials with infinitesimal distributions   $(\mu,\mu')$ and $(\nu,\nu')$, respectively. 
Then the finite-free additive convolution $\mfp \boxplus \mfq = (p_d\boxplus_d q_d)_{d\geq1}$ has infinitesimal distribution $(\rho,\rho')$ where $\rho=\mu\boxplus\nu$ and $\rho'$ has Cauchy transform $G_{\rho'}$ given by 
\begin{equation}\label{eq:Gprime sum}
G_{\rho'} =  -  (\widehat{R}_{\mfp} \circ G_{\mu \boxplus \nu} ) \cdot G'_{\mu \boxplus \nu}
    -  (\widehat{R}_{\mfq} \circ G_{\mu \boxplus \nu} ) \cdot G'_{\mu \boxplus \nu}
-
    \gdos_{\mu \boxplus \nu} \, .
\end{equation}
\end{theorem}
\begin{proof}
This is a consequence of Theorem \ref{thm:infinitesimal_additive_finite_free cumulants} and Lemma \ref{lemma:additive.delta.zero}. 
Indeed, From Lemma \ref{lem:infinitesimal_single_polynomial}, we have that our hypothesis are equivalent to the existence of measures $\mu$ and $\nu$ and sequences $(\widehat{r}_n(\mfp))_{n\geq 1}$ and $(\widehat{r}_n(\mfq))_{n\geq 1}$ such that
\begin{equation*}
\kappa_n(p_d)= r_n(\mu)+ \frac{1}{d}  \, \widehat{r}_n (\mfp)+o(1/d) \quad \text{and} \quad 
\kappa_n(q_d)= r_n(\nu)+ \frac{1}{d}  \, \widehat{r}_n (\mfq)+o(1/d) \quad 
\quad \forall \ n \geq 1.
\end{equation*} 
Moreover, the cumulants $\kappa_n(\, \cdot \,)$ and $r_n(\, \cdot \,)$ linearize the additive convolutions $p_d \boxplus q_d$ and $\mu \boxplus \nu$, respectively, 
so we obtain that 
\begin{equation*}
\kappa_n(p_d \boxplus q_d ) 
= r_n(\mu \boxplus \nu)+ \frac{1}{d}  \, \Big( \widehat{r}_n (\mfp) + \widehat{r}_n (\mfq) \Big) + o(d^{-1})  
\quad \forall \ n \geq 1. 
\end{equation*}
This means that $\widehat{r}_n (\mfp \boxplus \mfq) = \widehat{r}_n (\mfp) + \widehat{r}_n (\mfq)$, or equivalently, 
\begin{equation}\label{eqn:linealization_R_hat}
\widehat{R}_{\mfp \boxplus \mfq} = \widehat{R}_{\mfp} + \widehat{R}_{\mfq} \, .
\end{equation}
%
Then, by Theorem \ref{thm:infinitesimal_additive_finite_free cumulants} and Lemma \ref{lemma:additive.delta.zero}, we have that $\mfp \boxplus \mfq = (p_d\boxplus_d q_d)_{d\geq1}$ has infinitesimal distribution $(\rho,\rho')$ where $\rho=\mu\boxplus\nu$ and $\rho'$ has Cauchy transform $G_{\rho'}$ given by  
\begin{equation*} 
G_{\rho'} =
    -  (\widehat{R}_{\mfp \boxplus \mfq} \circ G_{\mu \boxplus \nu} ) \cdot G'_{\mu \boxplus \nu} -
    \gdos_{\mu \boxplus \nu}  
    =  -  (\widehat{R}_{\mfp} \circ G_{\mu \boxplus \nu} ) \cdot G'_{\mu \boxplus \nu} 
 -  (\widehat{R}_{\mfq} \circ G_{\mu \boxplus \nu} ) \cdot G'_{\mu \boxplus \nu} -
    \gdos_{\mu \boxplus \nu} \, . 
\end{equation*}    
\end{proof}

As a particular case, when one sequence of polynomials converges to $\delta_0$ in distribution, then the finite-free additive convolution and the infinitesimal free additive convolution yield the same distribution. 

\begin{corollary}
\label{cor:additive.delta.zero}
Suppose that $\mfp=(p_d)_{d\geq1}$ and $\mfq=(q_d)_{d\geq1}$ are two sequences of polynomials with infinitesimal distributions $(\mu,\mu')$ and $(\nu,\nu')$, respectively.
Let $(\gamma,\gamma')$ denote the infinitesimal free additive convolution $(\mu,\mu')\boxplus_B(\nu,\nu')$. 
If $\nu = \delta_0$, then the finite-free additive convolution $\mfp \boxplus \mfq = (p_d\boxplus_d q_d)_{d\geq1}$ has infinitesimal distribution $(\gamma,\gamma')$. 
\end{corollary}

\begin{proof}
Assume that $\nu = \delta_0$. 
This implies that $\mu \boxplus \nu = \mu$. 
Thus, Theorem \ref{thm:infinitesimal_additive_finite_free} gives that $\mfp \boxplus \mfq = (p_d\boxplus_d q_d)_{d\geq1}$ has infinitesimal distribution  $(\rho,\rho')$ where $\rho=\mu\boxplus\nu$ and $\rho'$ has Cauchy transform $G_{\rho'}$ given by 
\begin{equation}\label{eq:auxiliarGprime}
G_{\rho'} =  -  (\widehat{R}_{\mfp} \circ G_{\mu} ) \cdot G'_{\mu}
    -  (\widehat{R}_{\mfq} \circ G_{\mu} ) \cdot G'_{\mu}
-
    \gdos_{\mu} \,
=G_{\mu'} - (\widehat{R}_{\mfq} \circ G_{\mu} ) \cdot G'_{\mu},
\end{equation}
where the second equality follows from Lemma \ref{lemma:additive.delta.zero}. 

Now, since $(\gamma,\gamma')=(\mu,\mu')\boxplus_B(\nu,\nu')$, we know from \eqref{eqn:Cauchy_Rtransforms_infinitesimal_additive_preli} that $\gamma=\mu \boxplus \nu$ and $\gamma'$ is determined by its Cauchy transform 
\begin{equation*}
G_{\gamma'} = - ( \rinf_{\mu,\mu'} \circ G_{\mu \boxplus \nu}  ) \cdot G'_{\mu \boxplus \nu}
- ( \rinf_{\nu,\nu'} \circ G_{\mu \boxplus \nu}  ) \cdot G'_{\mu \boxplus \nu} \, .
\end{equation*}
Hence, to conclude that $(\rho,\rho')$ and $(\gamma,\gamma')$ are the same, it suffices to show that $G_{\rho'} = G_{\gamma'}$. 
To this end, note that the last relation reduces to 
\begin{equation*}
G_{\gamma'} = - ( \rinf_{\mu,\mu'} \circ G_{\mu}  ) \cdot G'_{\mu}
- ( \rinf_{\nu,\nu'} \circ G_{\mu}  ) \cdot G'_{\mu} ,
\end{equation*}
since $\nu = \delta_0$. 
But this is precisely the right-hand side of \eqref{eq:auxiliarGprime} since $G_{\mu'} = - ( \rinf_{\mu,\mu'} \circ G_{\mu}  ) \cdot G'_{\mu}$  due to \eqref{eqn:Mingo.relation} and $\widehat{R}_{\mfq}=\rinf_{\nu,\nu'}$ due to $H_{\nu} = 0$ because $\nu = \delta_0$, see Lemma \ref{lem:infinitesimal_single_polynomial_2}. 
\end{proof}

We end this part by describing, in the additive case, a general connection between polynomial and infinitesimal convolutions that exploits the theory of subordination functions.

\begin{theorem}
\label{prop:fluctuations_infinitesimal_subordination}
Suppose that $\mfp=(p_d)_{d\geq1}$ and $\mfq=(q_d)_{d\geq1}$ are two sequences of polynomials with infinitesimal distributions $(\mu,\mu')$ and $(\nu,\nu')$, respectively.
Let $(\gamma,\gamma')$ denote the infinitesimal free additive convolution $(\mu,\mu')\boxplus_B(\nu,\nu')$. 
Then the finite-free additive convolution $\mfp \boxplus \mfq = (p_d\boxplus_d q_d)_{d\geq1}$ has infinitesimal distribution $(\rho,\rho')$ where $\rho=\mu\boxplus\nu$ and $\rho'$ has Cauchy transform $G_{\rho'}$ given by 
\begin{equation}
\label{eq:Gprime_subordination}
G_{\rho'} = G_{\gamma'}
+\gdos_{\mu \boxplus \nu}  +\frac{\omega''_1}{2\omega'_1}+\frac{\omega''_2}{2\omega'_2} ,
\end{equation}
where $\omega_1$ and $\omega_2$ are the subordination functions determined by \eqref{eqn:def.subordination}. 
\end{theorem}

\begin{proof}
By Theorem \ref{thm:infinitesimal_additive_finite_free}, it suffices to show that  
$ -  (\widehat{R}_{\mfp} \circ G_{\mu \boxplus \nu} ) \cdot G'_{\mu \boxplus \nu} -  (\widehat{R}_{\mfq} \circ G_{\mu \boxplus \nu} ) \cdot G'_{\mu \boxplus \nu} - \gdos_{\mu \boxplus \nu}$ equals the right-hand side of \eqref{eq:Gprime_subordination}.  
For this, note first that $G_{\mu\boxplus \nu}=G_\mu\circ\omega_1$ implies  
\begin{equation*}
- (\widehat{R}_{\mfp} \circ G_{\mu \boxplus \nu} ) \cdot G'_{\mu \boxplus \nu}
=
	- (\widehat{R}_{\mfp} \circ G_\mu\circ\omega_1) \cdot ( G'_\mu\circ\omega_1) \cdot \omega'_1
=
	((G_{\mu'}+\gdos_{\mu})\circ\omega_1) \cdot \omega'_1,
\end{equation*}
where the second equality follows from the fact that $-(\widehat{R}_{\mfp} \circ G_\mu) \cdot  G'_\mu=G_{\mu'}+\gdos_{\mu}$, see \eqref{eq. qis0}.  
Then, taking $G_1=G_\mu$, $J=\omega_1$, and $G_2=G_{\mu\boxplus \nu}$, Lemma \ref{lem:H.circ.function} yields  
\begin{equation*}
(\gdos_{\mu}\circ\omega_1) \cdot \omega'_1=\gdos_{\mu \boxplus \nu} +\frac{\omega''_1}{2\omega'_1} \, .
\end{equation*}
Putting the last two equalities together, we obtain that 
\begin{equation*}
- (\widehat{R}_{\mfp} \circ G_{\mu \boxplus \nu} ) \cdot G'_{\mu \boxplus \nu}
= 
		(G_{\mu'}\circ\omega_1) \cdot \omega'_1 
	+ 
		\gdos_{\mu \boxplus \nu}+\frac{\omega''_1}{2\omega'_1}  \, .  
\end{equation*}
A similar argument gives that 
\begin{equation*}
-(\widehat{R}_{\mfq} \circ G_{\mu \boxplus \nu} ) \cdot G'_{\mu \boxplus \nu}
= 
		(G_{\nu'}\circ\omega_2) \cdot \omega'_2 
	+ 
		\gdos_{\mu \boxplus \nu} +\frac{\omega''_2}{2\omega'_2}.
\end{equation*}
Consequently, we have that 
\begin{align*} 
&-  (\widehat{R}_{\mfp} \circ G_{\mu \boxplus \nu} ) \cdot G'_{\mu \boxplus \nu} -  (\widehat{R}_{\mfq} \circ G_{\mu \boxplus \nu} ) \cdot G'_{\mu \boxplus \nu} - \gdos_{\mu \boxplus \nu}
\\ & \quad =  
		(G_{\mu'}\circ\omega_1) \cdot \omega'_1  
	+  
		(G_{\nu'}\circ\omega_2) \cdot \omega'_2
	+
		\frac{\omega''_1}{2\omega'_1}
	+
		\frac{\omega''_2}{2\omega'_2} 
	+
		\gdos_{\mu \boxplus \nu}.
\end{align*}
The conclusion follows from Proposition \ref{prop:infintesimal.subordination}, that asserts $G_{\gamma'}=(G_{\mu'}\circ\omega_1) \cdot \omega'_1  +  (G_{\nu'}\circ\omega_2) \cdot \omega'_2$. 
\end{proof}

\subsection{Multiplicative convolution}

\begin{theorem}
\label{prop:cauchy_multiplicative_finite_free}
	Suppose that $\mfp=(p_d)_{d\geq1}$ and $\mfq=(q_d)_{d\geq1}$ are two sequences of polynomials with infinitesimal distributions   $(\mu,\mu')$ and $(\nu,\nu')$, respectively. 
	Then the finite-free multiplicative convolution $\mfp \boxtimes \mfq = (p_d\boxtimes_d q_d)_{d\geq1}$ has  infinitesimal distribution $(\rho,\rho')$ where $\rho=\mu\boxtimes\nu$ and $\rho'$ has Cauchy transform $G_{\rho'}$ determined by 
	\begin{equation}\label{eq:Gprime multiplication}
		G_{\rho'} =  -  (\widehat{R}_{\mfp \boxtimes \mfq} \circ G_{\mu \boxtimes \nu} ) \cdot G'_{\mu \boxtimes \nu}
		-
		\gdos_{\mu \boxtimes \nu} 
	\end{equation}
with $\widehat{R}_{\mfp \boxtimes \mfq} (z) =  \sum_{n=1}^{\infty} \widehat{r}_{n}(\mfp \boxtimes \mfq) \, z^{n-1}$ and each $\widehat{r}_{n}(\mfp \boxtimes \mfq)$ given by \eqref{eqn:rhats_multiplicative}. 
\end{theorem}

\begin{proof}
By Proposition \ref{prop:finiteAsymptotics_multiplicative}, we know that $\mfp \boxtimes \mfq = (p_d\boxtimes_d q_d)_{d\geq1}$ that has limiting distribution $\mu \boxtimes \nu$.  
Moreover, by Theorem \ref{thm:infinitesimal_multiplicative_finite_free}, we have that 
\begin{equation*}
	\kappa_n(p_d \boxtimes_d q_d)= r_n(\mu \boxtimes \nu)+ \frac{1}{d}\widehat{r}_n ( p \boxtimes q ) + o(1/d) \qquad \text{for all } n\in \nn.
\end{equation*}
with each $\widehat{r}_{n}(\mfp \boxtimes \mfq)$ given by   \eqref{eqn:rhats_multiplicative}. 
The desired conclusion is then an immediate  consequence of Lemma \ref{lemma:additive.delta.zero}. 
\end{proof}

Similarly to the additive case, we can relate our results with infinitesimal freenees. 
The analogue in the multiplicative scenario corresponds to one of the sequences converging to $\delta_1$  in distribution.
\begin{proposition}\label{prop:finite_free_infinitesimal_coincide_mult}
Suppose that $\mfp=(p_d)_{d\geq1}$ and $\mfq=(q_d)_{d\geq1}$ are two sequences of polynomials with infinitesimal distributions $(\mu,\mu')$ and $(\nu,\nu')$, respectively.
Let $(\gamma,\gamma')$ denote the infinitesimal free multiplicative convolution $(\mu,\mu')\boxtimes_B(\nu,\nu')$. 
If $\nu = \delta_1$, then the finite-free multiplicative convolution $\mfp \boxtimes\mfq = (p_d\boxtimes_d q_d)_{d\geq1}$ has infinitesimal distribution $(\gamma,\gamma')$. 
\end{proposition}
\begin{proof}
From Definition \ref{eqn:infinitesimal_moments_of_a_product} and \eqref{def:infinitesimal_multiplicative_convolution}, 
we know that $\gamma=\mu \boxtimes \nu$ and $\gamma'$ is determined by its moments
\begin{equation}\label{eqn:infinitesimal_moments_of_a_product_2}
m_n(\gamma') 
  \ \  = 
        \sum_{\pi\in NC(n)}\left(
            \sum_{V\in\pi} 
                m_{|V|}(\mu') \, m_{\pi\backslash V}(\mu) 
                \cdot 
                r_{Kr(\pi)}(\nu) 
    +
        \sum_{W\in Kr(\pi)}
                m_\pi(\mu)
                \cdot 
                r'_{\abs{W}}(\nu,\nu') \,  r_{Kr(\pi)\backslash W}(\nu) \right)
\quad \forall n \geq 1
            .
\end{equation}
Hence, by Theorem \ref{prop:moments_for_multiplicative} and Theorem \ref{prop:cauchy_multiplicative_finite_free}, it suffices to show that $m_n(\gamma') = m'_n(\mfp \boxtimes \mfq ) $ for every $n \geq 1$.  
For this, note first that $\nu = \delta_1$ implies $r_n(\nu) = 0$ if $n\geq 2$. 
Moreover, every permutation $\sigma \in S_{NC}(t,s)$ with $t+s = n$ contains at least one block with two or more elements. 
Consequently, no permutation $\sigma$ contributes to the right-hand side of  \eqref{eqn:mprime_multiplicative}. 
Now, we have that $H_{\nu} = 0$ since $\nu = \delta_1$, and we then get     $\widehat{R}_{\mfq}=\rinf_{\nu,\nu'}$ from Lemma \ref{lem:infinitesimal_single_polynomial_2}, i.e.,  
$\widehat{r}_n(\mfq)=r'_n(\nu,\nu')$ for every $n\geq 1$. 
It follows that \eqref{eqn:mprime_multiplicative} and \eqref{eqn:infinitesimal_moments_of_a_product_2} since $m_n(\mu')$ is simply another notation for $m'_n(\mfp)$. 
\end{proof}

In the multiplicative case, we obtain a reformulation of the infinitesimal distribution from the previous proposition when considering subordination functions. 

\begin{corollary} \label{Cor. infmult}
Suppose that $\mfp=(p_d)_{d\geq1}$ and $\mfq=(q_d)_{d\geq1}$ are two sequences of polynomials with infinitesimal distributions $(\mu,\mu')$ and $(\nu,\nu')$. 
If $\mu\neq\delta_0$ and $\nu = \delta_1$, then the finite-free multiplicative convolution $\mfp \boxtimes\mfq = (p_d\boxtimes_d q_d)_{d\geq1}$ has infinitesimal distribution $(\gamma,\gamma')$ where  $\gamma=\mu \boxtimes \nu$ and 
\begin{equation}\label{eqn:multiplicative_infinitesimal}
 G_{\gamma'} 
    =
        G_{\mu'}
    +
        G_{\nu'} (\theta_{\mu}) \cdot \theta'_{\mu}
\text{\qquad with \qquad}
\theta_{\mu}(z) =
\frac{G_{\mu}(z)}{ G_{\mu}(z) - 1/z}
\, .
 \end{equation}
\end{corollary}
\begin{proof}


From Proposition \ref{prop:finite_free_infinitesimal_coincide_mult}, we know that  $(\gamma,\gamma')=(\mu,\mu')\boxtimes_B(\nu,\nu')$. 
This implies $\gamma = \mu \boxtimes \nu$. 
Moreover, from Proposition \ref{prop:infintesimal.subordinationmulti}, we get that 
\begin{equation}\label{eqn:multiplicative_infinitesimal_cauchyrep}
G_{\gamma'}
		=
			(G_{\mu'} \circ\tau_1)\tau'_1+(G_{\nu'}\circ\tau_2)\tau'_2
\quad \text{with} \quad 
\tau_{i} (z)  = \tfrac{1}{\omega_{i}(1/z)} 
\end{equation} 
and $\omega_{1},\omega_{2}$ determined by \eqref{MultSub}. 
Now, since $\nu=\delta_1$, we have that  $\psi_{\mu \boxtimes \nu} = \psi_{\mu}$ and $\psi_\nu (z) = \frac{z}{1-z}$. 
Consequently, \eqref{MultSub} yields
\begin{equation*}
\omega_{1} (z) = z
\text{\qquad and \qquad}
\omega_{2} (z) = \frac{\psi_{\mu}(z)}{\psi_{\mu}(z)+1} \, .
\end{equation*}
This implies $\tau_{1}(z) = z$ and $\tau_2(z) = \frac{1+\psi_{\mu}(1/z)}{\psi_{\mu}(1/z)} = \frac{G_{\mu}(z)}{ G_{\mu}(z) - 1/z}$. 
Therefore, \eqref{eqn:multiplicative_infinitesimal_cauchyrep} becomes  \eqref{eqn:multiplicative_infinitesimal}. 
\end{proof}

\begin{remark}
From a combinatorial point of view, the expression \eqref{eqn:multiplicative_infinitesimal} is equivalent to  
\begin{equation}\label{eqn:cauchy_multiplicative_infinitesimal_comb}
G_{\gamma'} =
{G_{\mu'}} + G_{\mfq,\mu} 
\text{\qquad with \qquad}
G_{\mfq,\mu}  = \sum_{n=1}^{\infty} \left( \sum_{\pi \in CI(n)} m_{\pi}(\mu)  \, r'_{\blocks{\pi}}(\nu,\nu')  \right)   z^{-n-1}
 \end{equation}
where $CI(n)$ denotes the set of all cyclic partitions of $[n]$. 
Indeed, since $Kr:NC(n) \to NC(n)$ is a bijection, \eqref{eqn:infinitesimal_moments_of_a_product} can be rewritten as
\begin{equation*}
m_n( \gamma' )
 = 
		\sum_{\pi\in NC(n)} 
		\sum_{V\in\pi}
			\Big(
					 m_{|V|}( \mu ) \, m_{\pi \backslash V}(\mu) \cdot r_{Kr(\pi)}(\nu)
				+
					m_{{Kr}^{-1}(\pi)}(\mu) \cdot 
					r'_{|V|}( \nu, \nu') \, r_{ \pi \backslash V}(\nu) 
			\Big)  \, .
\end{equation*}
But we know that $r_n(\nu)=0$ for $n \geq 2$ since $\nu=\delta_1$. 
Thus, we have that $r_{Kr(\pi)}(\nu)=0$ unless $\pi = 1_n$;
and that $r_{\pi \setminus V}(\nu)=0$ unless $\pi\setminus V$ contains only singletons as blocks.  
Consequently, the last expression reduces to  
\begin{equation*}
m_n(\gamma')
   \ \ \ =    \ \ \
        m_n(\mu')
       \ \ \ +
        \sum_{ \emptyset \neq S \subseteq [n] } 
            m_{{Kr}^{-1}(S \cup 0_{n\setminus S})}(\mu)
			\cdot
        	r'_{\abs{S}}(\nu,\nu')  \, 
             r_{0_{n\setminus S}}(\nu)			
\end{equation*}
where we have used the fact that $S \mapsto (0_{n\setminus S} \cup S,S)$ is a bijection from nonempty sets $S \subseteq [n]$ to pairs $(\pi,V)$ where $V \in \pi \in NC(n)$ and $ \pi\setminus V$ contains only singletons. 
The map  $S \mapsto {Kr}^{-1}(S \cup 0_{n\setminus S})$ is a bijection from nonempty sets $S \subseteq [n]$ to the set of cyclic interval partitions $CI(n)$, and this bijection satisfies that $\abs{S}= \blocks{\pi}$ whenever $\pi = {Kr}^{-1}(S \cup 0_{n\setminus S})$. 
Hence, we obtain that 
\begin{equation*}
m_n( \gamma' )
   \ \ \  =    \ \ \
        m_n(\mu')
       \ \ \ 
+
        \sum_{\pi \in CI(n)}
		m_{\pi} (\mu)
		\cdot
        r'_{\blocks{\pi}}(\nu,\nu') \, .
\end{equation*}
Finally, since $G_{\gamma'} = \sum_{n=1}^{\infty} m_n(\gamma') \, z^{-n-1}$ and $G_{\mu'} = \sum_{n=1}^{\infty} m_n(\mu') \, z^{-n-1}$, we obtain \eqref{eqn:cauchy_multiplicative_infinitesimal_comb}. 
\end{remark}


\section{Applications and Examples}\label{sec:applications_examples}

In this section, we apply our main results to compute infinitesimal distributions for several sequences of polynomials. 
Through the entire section, we consider sequences of polynomials $\mfp=(p_d)_{d\geq1}$ and $\mfq=(q_d)_{d\geq1}$  with infinitesimal distributions $(\mu,\mu')$ and $(\nu,\nu')$, respectively.
Then, along the way, we determine their infinitesimal distribution under concrete or additional hypothesis, such as specific classes of polynomials or polynomials with either exact finite-free cumulants or zero infinitesimal moments. 
If appropriate, we identify $\mu'$ and $\nu'$ with signed measures on $\rr$. 
%
%
%

\subsection{Basic Examples}

We first show that our results recover the infinitesimal distribution of sequences of polynomials $\mfp=(p_d)_{d\geq1}$ studied in \cite{arizmendi2023finite}, concretely, the case when each finite-free cumulant $\kappa^{(d)}_n(p_d)$ coincides with the free  cumulant $r_n(\mu)$ up to a correction term of order $o(1/d)$. 

\begin{example}[Zero cumulant fluctuations]
\label{exm:lim.dist.fixed.cumulants}
Assume that $\mfp=(p_d)_{d\geq1}$ has zero cumulant fluctuations, namely, its finite-free cumulants satisfy $\kappa^{(d)}_n(p_d) = r_n(\mu)+ o(1/d)$ for any integers $1 \leq n \leq d$. 
Note that this is equivalent to $\widehat{R}_\mfp(z)=0$.  
Thus, from Lemma \ref{lemma:additive.delta.zero}, we obtain that 
\begin{equation*}
 G_{\mu'} = -\gdos_\mu = \frac{G_{\mu}''}{2G_{\mu}'} - \frac{G_{\mu}'}{G_{\mu}} . 
\end{equation*}
This example recovers Theorem 1.6 from \cite{arizmendi2023finite}.
\end{example}

It is natural to examine the analogue case for moments, that is, each finite moment $m_n(p_d)$ coincides with $m_n(\mu)$ up to a correction term of order $o(1/d)$. 
In this case, the infinitesimal distribution is trivial since we have $m_n(\mu') = 0$, and hence $G_{\mu'}(z)=0$. 
However, it is still worth considering to establish $\widehat{R}_\mfp (z)$, which will be needed for further examples involving finite-free convolutions. 

\begin{example}[Zero infinitesimal moments]
\label{exm:lim.dist.fixed.moments}
Assume that $\mfp=(p_d)_{d\geq1}$ has zero infinitesimal moments, namely, $m_n(p_d) = m_n(\mu)+ o(1/d)$ for any integers $1 \leq n\leq d$. 
Thus, we have that $G_{\mu'}(z)=0$, and hence  $\rinf_{\mu,\mu'}=-(G_{\mu'}\circ K_\mu) \, \cdot K_\mu'=0$. 
Thus, by Lemma \ref{lem:infinitesimal_single_polynomial_2} and Lemma \ref{lem:infinitesimal_single_polynomial_3}, we obtain that 
\begin{equation*}\label{eqn:r_hat_p_with_no_infinitesimal}
\gdos_{\mu}= - (\widehat{R}_{\mfp} \circ G_{\mu} ) \cdot G'_{\mu} 
\,  \qquad \text{and} \qquad 
\widehat{R}_\mfp =  - (\gdos_{\mu} \circ K_{\mu} ) \cdot K'_{\mu} = -\frac{ \ K''_\mu(z) \ }{2K'_\mu(z)}-\frac{1}{z} \, . 
\end{equation*}
\end{example}

The following is the case where all but finitely many roots are located at a single value. 
As in the previous example, a main goal is determining $\widehat{R}_\mfp (z)$ for later use on finite-free convolutions.

\begin{example}[Finite discrete deviation from Dirac distribution]
\label{exm:lim.dist.finite.pertubation}
Assume that each polynomial $p_d$  from the sequence $\mfp=(p_d)_{d\geq1}$ is given by 
\begin{equation*}
p_d \, (x) =(x-\alpha)^{d-s}(x-\alpha_1)(x-\alpha_2) \cdots (x-\alpha_s) 
\end{equation*}
for a fixed integer $s \geq 1$ and fixed real numbers $\alpha, \alpha_1,\dots, \alpha_s\in \mathbb{R}$. 
The $n$-th moment of $p_d$ is given by
\begin{equation*}
m_n(p_d)
    = 
        \alpha^n + \frac{1}{d} \Big(-s\alpha^n + \sum_{k=1}^s \alpha_k^n \Big) 
    = 
        m_n (\mu) +\frac{1}{d} m_n(\mu') \, . 
\end{equation*}
This implies that the infinitesimal distribution 
$(\mu,\mu')=(\delta_\alpha, -s \delta_\alpha + \sum_{k=1}^s \delta_{\alpha_k})$. 
Hence, the Cauchy transforms of $\mu$ and $\mu'$ are respectively given by
\begin{equation*}
G_{\mu}(z) =  \frac{1}{z-\alpha}= \frac{1}{F_\mu (z)}
\qquad \text{and} \qquad
G_{\mu'} (z) =  - \frac{s}{z-\alpha} + \sum_{k=1}^{s} \frac{1}{z-\alpha_k} \, . 
\end{equation*}
Moreover, from the Cauchy transform of $\mu$, we obtain that 
\begin{equation*}
K_{\mu}(z) = G^{<-1>}(z)  = \frac{1}{z} + \alpha, 
\qquad 
K'_{\mu}(z) =  - \frac{1}{z^2},  
\qquad \text{and} \qquad 
\gdos_{\mu}(z) = \frac{1}{2} \frac{d}{dz} \ln(F'_\mu) = 0 \, .
\end{equation*}
Thus, since $ \gdos_{\mu}(z)=0$, Lemma \ref{lem:infinitesimal_single_polynomial_2} yields   
\begin{equation} 
G_{\mu'}= - (\widehat{R}_\mfp \circ G_{\mu}) \cdot G'_{\mu} 
\qquad \text{ and }\qquad 
\widehat R_{\mfp} = -(G_{\mu'}\circ K_{\mu}) \cdot K'_{\mu} \, .
\end{equation}
Notice that the last two equations imply that $\widehat{R}_\mfp=\rinf_{\mu,\mu'}$, 
so the cumulant fluctuations of $\mfp=(p_d)_{d\geq1}$ coincide exactly with the infinitesimal cumulants of $(\mu,\mu')$  in this case. 
The function $\widehat{R}_\mfp(z)$ is explicitly given by 
\begin{align}
\widehat{R}_\mfp(z) 
 &= 
		 -(G_{\mu'} (  K_{\mu} (z) ) \cdot  K'_{\mu}(z)
     = 
        \frac{1}{z^2} \cdot 
        G_{\mu'}\left( {\textstyle \frac{1+\alpha z}{z}} \right) 
    =       \frac{1}{z^2} \cdot         \left( - \frac{s}{\frac{1+\alpha z}{z}-\alpha} + \sum_{k=1}^{s} \frac{1}{\frac{1+\alpha z}{z}-\alpha_k}  \right)   \nonumber \\
   & =        \frac{1}{z} \cdot         \left( -s  + \sum_{k=1}^{s} \frac{1}{ 1+ \alpha z -\alpha_k z}  \right) 
    =
        \frac{1}{z} \cdot
        \sum_{k=1}^{s} \frac{-\alpha z +\alpha_k z}{ 1+ \alpha z -\alpha_k z}  
    =
        \frac{1}{z} \cdot
        \sum_{k=1}^{s} \frac{ (\alpha_k-\alpha) z}{ 1 - (\alpha_k-\alpha) z}. 
\end{align}
Therefore, from the geometric series of each $(\alpha_k-\alpha) z$, we conclude that 
\begin{equation} 
\widehat{R}_\mfp(z)
    =
        \sum_{n=1}^{\infty} 
		\sum_{k=1}^{s} 
				(\alpha_k-\alpha)^n z^{n-1} 
\qquad \text{ and } \qquad 
\widehat{r}_n( \mfp )
\ 	=
		\sum_{k=1}^{s}
				(\alpha_k-\alpha)^n.   
\end{equation}
It is worth noticing that this can also be established through finite-free additive  convolution. 
Indeed, having determining $\widehat{R}_\mfp(z)$ for $\alpha=0$, the general case would follow from $\widehat{R}_{\mfp \boxplus \mfq}$ with $q_d=(x-\alpha)^d$. 
\end{example}

\subsection{Differentiation}

In this part, we analyze the change in infinitesimal distribution after repeated differentiation of polynomials.

\begin{proposition}[Repeated differentiation]
\label{prop:repeated.diff.general}
 Let $t\in(0,1]$ and $\alpha\in\rr$ and construct as sequence of degrees $d_i$ and values $1\leq j_i \leq d_i$ such that $\lim_{i\to\infty} d_i=\infty$ and $\frac{j_i}{d_i}=t+\frac{\alpha}{d_i}+ o(1/d_i)$. Let $\mfp=(p_i)_{i=1}^\infty$ be a sequence of polynomials such that $p_i$ has degree $d_i$ and with limiting distribution $(\mu,\mu')$ and consider the polynomial $q_i= p^{(d_i-j_i)}_{i}$ obtained after differentiating $d_i-j_i$ times the polynomial $p_{i}$. Then the sequence $\mfq=(q_i)_{i=1}^\infty$ has infinitesimal distribution $(\nu,\nu')$. Where $\nu=\text{Dil}_{t} \mu^{\boxplus \frac{1}{t}}$ and the Cauchy transform of $\nu'$ is
\begin{equation}
\label{eq:infinitesimal.several.derivatives}
G_{\nu'} (z) =  -\frac{\alpha}{t} 
\left(G_{\nu}(z)+\frac{G'_{\nu}(z)}{G_{\nu}(z)} \right) -  \widehat{R}_{\mfp}\left( t G_{\nu}(z) \right) \cdot G'_{\nu} (z)- \gdos_{\nu}(z).
\end{equation}
\end{proposition}

\begin{proof}
Recall from \cite[Proposition 3.4]{arizmendi2024s} that the cumulants of $q_i$ are given by 
\[
\ffc{n}{j_i}(q_i) = \left( \frac{j_i}{d_i}\right)^{n-1} \ffc{n}{d_i}(p_i) =\left(t^{n-1}+\tfrac{\alpha (n-1)t^{n-2}}{d_i} + o(1/d_i)\right) \ffc{n}{d_i}(p_i)  \quad \text{ for } 1 \leq n \leq j_i \leq d_i. 
\]
Thus, \eqref{eqn:hypothesis_cumulants_ver_1} yields 
$$\ffc{n}{j_d}(q_d) = t^{n-1} r_n (\mu) + \frac{1}{d_i} \left(\alpha (n-1)t^{n-2}  r_n (\mu) + t^{n-1} \widehat{r}_n(\mfp)\right) + o(1/d_i).$$
This means that the cumulant fluctuations are $\widehat{r}_n(\mfq)=\alpha (n-1)t^{n-2}  r_n (\mu) + t^{n-1} \widehat{r}_n(\mfp)$. Notice that
$$\sum_{n=1}^\infty \alpha (n-1)t^{n-2}  r_n (\mu) z^{n-1} = 
\alpha z \sum_{n=2}^\infty (n-1) r_{n} (\mu) (tz)^{n-2} = \alpha z R'_{\mu}(tz)= \frac{\alpha z}{t} R'_\nu(z),$$
so we get that 
$$\widehat{R}_\mfq(z)= \alpha z R'_{\mu}(tz) + \widehat{R}_\mfp(tz)$$

Using \eqref{eq. qis0} we obtain that
$$G_{\nu'} =  -\frac{\alpha}{t} G_{\nu}(z) R'_\nu\left( G_{\nu}(z) \right) \cdot G'_{\nu} (z) -  \widehat{R}_{\mfp}\left( t G_{\nu}(z) \right) \cdot G'_{\nu} (z)- \gdos_{\nu}(z).$$
To simplify we notice that substituting $G_\nu$ in \textit{(2)} from Proposition \ref{prop:basic.funct.relation} and differentiating one has that
$$1 + \frac{G'_\nu}{G_\nu^2}= (R'_\nu\circ G_\nu)G'_\nu.$$
Using this in the previous equation, yields the desired result.
\end{proof}

\begin{corollary}[Repeated differentiation]
\label{cor:infinitesimal.several.derivatives.alpha.zero}
With the assumptions of Proposition \ref{prop:repeated.diff.general}, if $\alpha=0$, namely $\frac{j_i}{d_i}=t+ o(1/d_i)$, then the formula simplifies to 
\begin{equation}
\label{eq:infinitesimal.several.derivatives.alpha.zero}
G_{\nu'}(z) =-  \widehat{R}_{\mfp}\left( t G_{\nu}(z) \right) \cdot G'_{\nu} (z)- \gdos_{\nu}(z).
\end{equation}
\end{corollary}

\begin{corollary}
With the assumptions of Proposition \ref{prop:repeated.diff.general}, if $\mfp$ does not have cumulant fluctuations, namely $\widehat{R}_\mfp=0$, then formula \eqref{eq:infinitesimal.several.derivatives} simplifies to
\begin{equation}
G_{\nu'} =   -\frac{\alpha}{t} 
\left(G_{\nu}+\frac{G'_{\nu}}{G_{\nu}} \right) - \gdos_{\nu}.
\end{equation}
\end{corollary}

\begin{corollary}[One derivative]
Let $\mfp=(p_i)_{i=1}^\infty$ be a sequence of polynomials such that $p_i$ has degree $d_i$ tending to $\infty$ and with limiting distribution $(\mu,\mu')$ and consider the polynomial $q_i= p'_{i}$ obtained after differentiating once the polynomial $p_{i}$. Then the sequence $\mfq=(q_i)_{i=1}^\infty$ has infinitesimal distribution $(\nu,\nu')$. Where $\nu=\mu$ and 
\begin{equation}
\label{eq:infinitesimal.one.derivative}
\nu'=\mu'+\mu-\MM(\mu),
\end{equation}
where $\MM(\mu)$ is the inverse Markov-Krein transform of $\mu$.
\end{corollary}

\begin{proof}
Notice that if we differentiate once, then $\frac{d_i-1}{d_i}=1+\frac{-1}{d_i}+ o(1/d_i)$, so we are in the case $t=1$ and $\alpha=-1$. Therefore $\nu=\mu$ and \eqref{eq:infinitesimal.several.derivatives} yields
$$
G_{\nu'} = 
\left(G_{\mu}+\frac{G'_{\mu}}{G_{\mu}} \right) - \left( \widehat{R}_{\mfp}\circ G_{\mu} \right) \cdot G'_{\mu} - \gdos_{\mu}=G_{\mu}+\frac{G'_{\mu}}{G_{\mu}}+G_{\mu'}.
$$
By \cite[Equation(2.4.3)]{kerov1998}, the middle term can be identified as the negative of the inverse Markov-Krein transform $G_{\MM(\mu)}=-\frac{G'_{\mu}}{G_{\mu}}$ and the conclusion follows.
\end{proof}

\begin{example}[Differentiating Hermite Polynomials]
\label{ex.hermite}
Recall from \cite[Example 5.8]{arizmendi2023finite} that the Hermite polynomials have limiting distribution $(\mu,\mu')$ where $\mu=\mathfrak{s}$ is semicircular law and $\mu'=\tfrac{1}{2}\left(\mathfrak{a}-\mathfrak{b}\right)$ is the difference of a symmetric arcsine, $\mathfrak{a}$ and a Bernoulli distribution, $\mathfrak{b}$.

Recall from \cite[Example 5.2.7]{kerov1998}
that $\MM(\mathfrak{s})=\MM(\mathfrak{a})$. Therefore, if we differentiate the Hermite polynomials once, we obtain and infinitesimal distribution $(\nu,\nu')$ where $\nu=\mathfrak{s}$ is a semicircle law and by \eqref{eq:infinitesimal.one.derivative} we get that
 $$\nu'=\tfrac{1}{2}\left(\mathfrak{a}-\mathfrak{b}\right) +\mathfrak{s}-\mathfrak{a}=\mathfrak{s}-\tfrac{1}{2}\mathfrak{a}-\tfrac{1}{2}\mathfrak{b}.$$
\end{example}

\begin{example}[Differentiating Bernoulli Polynomials]
\label{ex.bernoulli}
Consider the sequence of polynomials $\mfp=(p_i)_{i=1}^\infty$ where $p_i(x):=(x-1)^i(x+1)^i$ has degree $d_i=2i$ with half of the roots at $1$ and half at $-1$. Since for all $i\in\nn$ the spectral measure $\meas{p_i}=\frac{1}{2}\delta_{1}+\frac{1}{2}\delta_{-1}$ is a symmetric Bernoulli, $\mathfrak{b}$, then $m_n(p_i) = m_n(\mathfrak{b})$ for all $n$, and then the sequence has zero infinitesimal moments. Thus, Example \ref{exm:lim.dist.fixed.moments} yields that $\mfp$ has infinitesimal distribution $(\mathfrak{b},\mu')$ with Cauchy transforms given by
$$G_{\mathfrak{b}}(z)=\frac{z}{(z-1)(z+1)}, \qquad \text{and} \qquad G_{\mu'}(z)=0.$$
One can check (see \cite[Example 12.8]{nica2006lectures}) that $K_{\mathfrak{b}}(z)=\frac{1+\sqrt{1+4z^2}}{2z}$, so 

$$K'_{\mathfrak{b}}(z)=-\frac{1+\sqrt{1+4z^2}}{2z^2\sqrt{1+4z^2}}, \qquad K''_{\mathfrak{b}}(z)=\frac{ 2(1+4z^2)^{3/2} + 3(1+4z^2)-1 }{2z^3(1+4z^2)^{3/2} }, $$
and
$$
\widehat{R}_\mfp =-\frac{ \ K''_{\mathfrak{b}}(z) \ }{2K'_{\mathfrak{b}}(z)}-\frac{1}{z} = \frac{\sqrt{1+4z^2}-1}{2z(1+4z^2)}.
$$

Consider now $t=\frac{1}{2}$ and construct the polynomial $q_i= p^{(i)}_{i}$ obtained after differentiating $i$ times the polynomial $p_{i}$ to obtain a polynomial of degree $i$. Then the sequence $\mfq=(q_i)_{i=1}^\infty$ has infinitesimal distribution $(\nu,\nu')$. We know (see for instance \cite[Example 14.15]{nica2006lectures}) that $\nu=\text{Dil}_{1/2} {\mathfrak{b}}^{\boxplus 2}=\mathfrak{a}$ is an arcsine distribution, with Cauchy transform
$$G_{\mathfrak{a}}(z)=\frac{1}{\sqrt{z^2-4}}.$$
So we obtain,
$$G'_{\mathfrak{a}}(z)=\frac{-z}{(z^2-4)^{3/2}}\qquad \text{and}\qquad H_{\mathfrak{a}}=\frac{2}{z(z^2-4)}= \frac{1}{4}\left(\frac{1}{z+2}+\frac{1}{z-2} \right) - \frac{1}{2z}.$$
Notice that $H_{\mathfrak{a}}$ is the Cauchy transform of the measure $\frac{1}{2}\mathfrak{b}_2-\delta_0$, a difference between a Bernoulli with atoms at 2 and $-2$ and and delta Dirac distribution with atom at 0.

In order to use Corollary \ref{cor:infinitesimal.several.derivatives.alpha.zero}, we compute
\begin{equation*}
 \widehat{R}_{\mfp}\left( \frac{1}
{2}G_{\mathfrak{a}}(z) \right) \cdot G'_{\mathfrak{a}} (z)=\frac{\sqrt{1+\frac{1}{z^2-4}}-1}{\frac{1}{(z^2-4)^{1/2}}(1+\frac{1}{z^2-4})} \frac{-z}{(z^2-4)^{3/2}} =-\frac{z}{\sqrt{z^2-4}\sqrt{z^2-3}} +\frac{z}{z^2-3}.
\end{equation*}
can identify the first term as the Cauchy transform obtained as the operation $\flat$ of two Bernoulli distributions. The operation $\flat$ was recently introduced in \cite[Lemma 3.2]{goldstein2024free} in connection to free zero bias. On the other hand, the second term is the Cauchy transform of a Beronulli distribution.

Using Corollary \ref{cor:infinitesimal.several.derivatives.alpha.zero} we obtain
\begin{equation*}
G_{\nu'} (z) =  -  \widehat{R}_{\mfp}\left( \frac{1}
{2}G_{\mathfrak{a}}(z) \right) \cdot G'_{\mathfrak{a}} (z)- \gdos_{\mathfrak{a}}(z) =\frac{z}{\sqrt{z^2-4}\sqrt{z^2-3}} -\frac{z}{z^2-3} -\frac{2}{z(z^2-4)}.
\end{equation*}
If we denote by $\mathfrak{b}_c:=\frac{1}{2}\delta_{-c}+ \frac{1}{2}\delta_{c}$ the symmetric Bernoulli distribution with atoms at $-c$ and $c$, then we can express $\nu'$ as 
$$\nu'= \mathfrak{b}_2\flat \mathfrak{b}_{\sqrt{3}} -  \mathfrak{b}_{\sqrt{3}} + \tfrac{1}{2}\delta_0 - \tfrac{1}{2} \mathfrak{b}_{2},$$
which is a signed measure with total mass $0$.
\end{example}

\subsection{Infinitesimal distribution under additive convolution}

We now turn to the application of our results to finite-free additive  convolutions of concrete sequences of polynomials. 
We start with an example that retrieves, as a particular cases, the distributions examined by Shlyakhtenko in \cite[Section 4.1]{shlyakhtenko2018free}.

\begin{example} 
[Additive convolution with finite perturbation]
\label{exm:deviation_from_dirac}
Suppose that each polynomial $q_d$ from the sequence $\mfq=(q_d)_{d\geq1}$ is given by 
\begin{equation*}
	q_d \, (x) =x^{d-t}(x-\beta_1)(x-\beta_2) \cdots (x-\beta_t) 
\end{equation*}
for a fixed integer $t \geq 1$ and fixed real numbers $ \beta_1,\dots, \beta_t\in \mathbb{R}$. 
Observe that this a special case of Corollary \ref{cor:additive.delta.zero} since $\mfq = (q_d)_{d\geq1}$ has indeed limiting root distribution equal to $\delta_0$. 
Thus, the sequence $\mfp \boxplus \mfq = (p_d\boxplus_d q_d)_{d\geq1}$ has infinitesimal distribution  $(\rho,\rho')$  where $\rho=\mu \boxplus \delta_0 = \mu$ and $\rho'$ has Cauchy transform 
	$
G_{\rho'} 
	=
		G_{\mu'}
	-
		(\widehat{R}_\mfq \circ G_{\mu}) \cdot G'_{\mu} 
$.  
Moreover, from Example \ref{exm:lim.dist.finite.pertubation}, the Cauchy transform $G_{\rho'} $ takes the form 
\begin{equation*}
	G_{\rho'} 
	= G_{\mu'}  
	- 		\sum_{k=1}^{t} \frac{\beta_k \, G'_\mu }{1-\beta_k \, G_{\mu} } \ .
\end{equation*}
Additionally, if $\mfp=(p_d)_{d\geq1}$ has zero infinitesimal moments, see Example \ref{exm:lim.dist.fixed.moments}, then $G_{\mu'} = 0$; 
and Corollary \ref{cor:additive.delta.zero} gives that the pair $(\rho,\rho')$ coincides with the infinitesimal free additive convolution of $(\mu,\mu') = (\mu,0) $ and $(\nu,\nu')=(\delta_0,-s \delta_0 + \sum_{k=1}^t \delta_{\beta_k})$. 
Consequently, this example retrieves the distributions studied by Shlyakhtenko in \cite[Section 4.1]{shlyakhtenko2018free}.  
In particular, if each $q_d$ has at most one nonzero root, i.e., $t=1$, 
then $G_{\rho'} $ simplifies to  
\begin{equation*}
	G_{\rho'} = 
	\frac{ \beta \, G'_{\mu} }{1 - \beta \, G_{\mu} }.
\end{equation*}
\end{example}

The previous example gives an instance where finite-free convolution and infinitesimal convolution yield the same distributions. 
In general, however, this is not the case, even if both $\mfp=(p_d)_{d\geq1}$ and $\mfq=(q_d)_{d\geq1}$ have zero infinitesimal moments. 
We establish this next.

\begin{example}[Additive convolution of polynomials with fixed moments] \label{ex.additivefixedmoments}
Suppose that both $\mfp=(p_d)_{d\geq1}$ and $\mfq=(q_d)_{d\geq1}$ have zero infinitesimal moments, see Example \ref{exm:lim.dist.fixed.moments},   
meaning that their infinitesimal distributions are $(\mu,\mu')=(\mu,0)$ and $(\nu,\nu')=(\nu,0)$, respectively, where we use 0 to denote the trivial measure that assigns zero to every measurable set, following the notation in \cite{shlyakhtenko2018free}.
It then follows from  Theorem \ref{thm:infinitesimal_additive_finite_free} and Example \ref{exm:lim.dist.fixed.moments} that $\mfp \boxplus \mfq = (p_d\boxplus_d q_d)_{d\geq1}$ has infinitesimal distribution  $(\rho,\rho')$  where 
$\rho = \mu \boxplus \nu$ and $\rho'$ has Cauchy transform $G_{\rho'}$ given by 
\begin{equation*}
	G_{\rho'} =   
	- \left( 
		\frac{K''_\mu(G_{\mu\boxplus\nu})}{2K'_\mu(G_{\mu\boxplus\nu})}+\frac{1}{G_{\mu\boxplus\nu}}
	+
		\frac{K''_\nu(G_{\mu\boxplus\nu})}{2K'_\nu(G_{\mu\boxplus\nu})}+\frac{1}{G_{\mu\boxplus\nu}}
	\right) \cdot G'_{\mu\boxplus\nu}
	- H_{\mu\boxplus\nu}.
\end{equation*}
In terms of subordination functions, see Theorem \ref{prop:fluctuations_infinitesimal_subordination}, and since $(\mu \boxplus \nu,0)$ is the infinitesimal free additive convolution of $(\mu,0)$ and $(\nu,0)$, the Cauchy transform $G_{\rho'} $ is also given by 
\begin{equation*}
	G_{\rho'} =\gdos_{\mu \boxplus \nu}+\frac{\omega''_1}{2\omega'_1}+\frac{\omega''_2}{2\omega'_2} .
\end{equation*}
\end{example}

We now address the case when both $\mfp=(p_d)_{d\geq1}$ and $\mfq=(q_d)_{d\geq1}$ have zero cumulant fluctuations. 

\begin{example}[Convolution with a sequence with zero cumulant fluctuations]
Suppose that $\mfp=(p_d)_{d\geq1}$ has zero cumulant fluctuations, see Example \ref{exm:lim.dist.fixed.cumulants}. 
Then, from Theorem \ref{thm:infinitesimal_additive_finite_free}, we obtain that $\mfp \boxplus \mfq = (p_d\boxplus_d q_d)_{d\geq1}$ has infinitesimal distribution  $(\rho,\rho')$  where $\rho=\mu\boxplus\nu$ and $\rho'$ has Cauchy transform $G_{\rho'}$ given by 
\begin{equation*}
	G_{\rho'} = -  (\widehat{R}_{\mfq} \circ G_{\mu \boxplus \nu} ) \cdot G'_{\mu \boxplus \nu}
	-\gdos_{\mu \boxplus \nu} .
\end{equation*}
Additionally, if $\mfq=(q_d)_{d\geq1}$ has also zero cumulant fluctuations, the last expression reduces to $G_{\rho'} = -\gdos_{\mu \boxplus \nu}$. 
\end{example}

The next example examines the case when both $\mfp=(p_d)_{d\geq1}$ and $\mfq=(q_d)_{d\geq1}$ have all but finitely many roots located at $0$. 
For this, we obtain that the infinitesimal moments of the convolution equal the sum of the individual infinitesimal moments.

\begin{example}[Additive convolution of finite perturbations]
\label{exm:deviation_from_dirac2}
Suppose that each $p_d$ and each $q_d$ from the sequences $\mfp=(p_d)_{d=s}^\infty$ and $\mfq=(q_d)_{d=t}^\infty$ are given by 
\begin{equation*}
	p_d(x)=x^{d-s}(x-\alpha_1)(x-\alpha_2) \cdots (x-\alpha_s) 
\qquad \text{and}\qquad 
	q_d(x)=x^{d-t}(x-\beta_1)(x-\beta_2) \cdots (x-\beta_t)
\end{equation*}
for fixed integers $s,t \geq 1$ and fixed real numbers  $\alpha_1,\dots, \alpha_s, \beta_1,\dots,\beta_t\in \mathbb{R}$.  
From Example \ref{exm:lim.dist.finite.pertubation}, we know that $\mfp=(p_d)_{d=s}^\infty$ and $\mfq=(q_d)_{d=t}^\infty$ have infinitesimal distributions $(\mu,\mu')=(\delta_0,-s\delta_0 + \sum_{k=1}^{s}\delta_{\alpha_k})$ and  $(\nu,\nu')=(\delta_0,-t\delta_0 + \sum_{k=1}^{t}\delta_{\beta_k})$, respectively. 
Then, from Theorem \ref{thm:infinitesimal_additive_finite_free}, we have that $\mfp \boxplus \mfq = (p_d\boxplus_d q_d)_{d\geq1}$ has infinitesimal distribution  $(\rho,\rho')$  where 
$\rho = \mu \boxplus \nu = \delta_0 $ and $\rho'$ has Cauchy transform $G_{\rho'}$ given by 
\begin{equation*}
G_{\rho'} =  
-  (\widehat{R}_{\mfp} \circ G^{}_{\delta_0} ) \cdot G'_{\delta_0} 
-  (\widehat{R}_{\mfq} \circ G^{}_{\delta_0} ) \cdot G'_{\delta_0}
+ \gdos^{}_{\delta_0}.
\end{equation*}
But, from  Example \ref{exm:lim.dist.finite.pertubation}, we also know that  $G_{\mu'}= - (\widehat{R}_\mfp \circ G^{}_{\delta_0}) \cdot G'_{\delta_0}$ and $G_{\nu'}= - (\widehat{R}_\mfq \circ G^{}_{\delta_0}) \cdot G'_{\delta_0}$ in this case. 
Hence, since $H^{}_{\delta_0}=0$, we obtain that 
\begin{equation*}
	G_{\rho'} 
=  
		G_{\mu'} +G_{\nu'} 
=	
{ \left( 	- \frac{s}{z} + \sum_{k=1}^{s} \frac{1}{z-\alpha_k} \right)
	+ \left(- \frac{t}{z} + \sum_{k=1}^{t} \frac{1}{z-\beta_k} \right) } , 
\end{equation*}
or equivalently, $	\rho' = \mu' + \nu' = -(s+t)\delta_0 + \sum_{k=1}^{s} \delta_{\alpha_k}+\sum_{k=1}^{t} \delta_{\beta_k}$. 
\end{example}
From Corollary \ref{cor:additive.delta.zero}, we know that finite-free convolution and infinitesimal free convolution yield the same distributions in the additive case whenever either $\mu = \delta_0$  or $\nu = \delta_0$. 
The previous example is a concrete case where our methods provide an alternative approach to computing infinitesimal free convolutions.

\subsection{Infinitesimal distribution under multiplicative convolution} 

Before we take on concrete examples of infinitesimal distributions induced by finite-free multiplicative convolution, we would like to stress that the computation of $\widehat{R}_{\mfp \boxtimes \mfq}$ is a main distinction from the additive case. 
Indeed, for multiplicative convolutions, we have the functional relation \eqref{eq:Gprime multiplication}, namely,  
\begin{equation*}
G_{\rho'} 
		=  -  (\widehat{R}_{\mfp \boxtimes \mfq} \circ G_{\mu \boxtimes \nu} ) \cdot G'_{\mu \boxtimes \nu}
		-
		\gdos_{\mu \boxtimes \nu},
\end{equation*}
which, since $\widehat{R}_{\mfp \boxplus \mfq} = \widehat{R}_{\mfp} + \widehat{R}_{\mfq}$, is entirely analogous to \eqref{eq:Gprime sum} in the additive case.  
However, while $\widehat{R}_{\mfp \boxplus \mfq}$ can be straightforwardly computed by adding $\widehat{R}_{\mfp}$ and $\widehat{R}_{\mfq}$, the computation of $\widehat{R}_{\mfp \boxtimes \mfq}$ seems to be more intricate. 
For instance, even if both $\mfp=(p_d)_{d\geq1}$ and $\mfq=(q_d)_{d\geq1}$ have zero cumulant fluctuations, meaning that 
$\widehat{r}_n(\mfp)= \widehat{r}_n(\mfq) = 0$  for all $n \geq 1$, 
then \eqref{eqn:rhats_multiplicative} reduces to  
\begin{equation}\label{eqn:annular_generating_question}
\widehat r_n(\mfp\boxtimes \mfq)
	= 
-	
		\frac{n}{2} 
		\sum_{\substack{ t+s = n \\ \sigma \in S_{NC}(t,s) } } 
				\frac{r_{\sigma}(\mu) \cdot r_{Kr_{t,s}(\sigma)}(\nu)}{ts} \, , 
\end{equation}
but, to the best of our knowledge, the generating function associated to coefficients satisfying the combinatorial relation \eqref{eqn:annular_generating_question} appears to be a non-trivial matter. 
Nonetheless, our results can still be successfully applied to obtain the infinitesimal distribution in several relevant scenarios. 
%

\begin{example}[Multiplicative convolution with Marchenko-Pastur] 
Assume that both $\mfp=(p_d)_{d\geq1}$ and $\mfq=(q_d)_{d\geq1}$ have zero cumulant fluctuations. 
If $\nu$ is the Marchenko-Pastur of parameter $1$, which is the distribution characterized by the condition that $r_n(\nu)=1$ for all $n \geq 1$, then \eqref{eqn:annular_generating_question} becomes 
\begin{equation*}
\widehat r_n(\mfp\boxtimes \mfq)
		= 
			-\frac{n}{2} 
				\sum_{\substack{ t+s = n \\ \sigma \in S_{NC}(t,s) } } 
				\frac{r_{\sigma}(\mu)}{ts} \, , 
\quad	 \text{or equivalently}, \quad 
\widehat R_{\mfp\boxtimes \mfq}=-H_{\mu} \, .
\end{equation*}
Hence, from Theorem \ref{prop:cauchy_multiplicative_finite_free}, the sequence $\mfp \boxtimes \mfq = (p_d\boxtimes_d q_d)_{d\geq1}$ has infinitesimal distribution  $(\rho,\rho')$  where $\rho=\mu\boxtimes\nu$ and $\rho'$ has Cauchy transform $G_{\rho'}$ given by 
\begin{equation*} 
G_{\rho'} = H_{\mu} \circ G_{\mu\boxtimes\nu}  \cdot G'_{\mu\boxtimes\nu} -H_{\mu\boxtimes\nu}. 
\end{equation*}
The Cauchy transform $G_{\mu\boxtimes\nu}$, and consequently, $G'_{\mu\boxtimes\nu}$ and $H_{\mu\boxtimes\nu}$, can be recovered from the functional equation
\begin{equation*}
z \, G_{\mu\boxtimes\nu}(z)^2 
=-G_\mu \left(\tfrac{-1}{G_{\mu\boxtimes\nu}(z)} \right) \, .
\end{equation*}
More concretely, we have that $\nu$ is the Marchenko-Pastur of parameter $1$ if each $q_d$ is given by 
\begin{equation*}
q_d(x) = \sum_{i=0}^d (-1)^i x^{d-i}  \binom{d}{i} \frac{(d)_{i} }{d^i} ,
\end{equation*}
which is called the $d$-th Laguerre polynomial of parameter $1$ and it satisfies that  $\kappa^{d}_n(  q_d ) = 1$ for $1\leq n \leq d$. 
\end{example}

We now consider the case when  $\mfq=(q_d)_{d\geq1}$ has all but finitely many roots located at $0$. 
In this case, the polynomials $p_d\boxtimes_d q_d$ have roots that concentrate at $0$ regardless of $p_d$ as $d \to \infty$, since $\mu \boxtimes \delta_0 = \delta_0$, 
with infinitesimal moments corresponding to those of a multiple and dilation of $\nu'$ by $m_1(\mu)$, in particular, they do not depend on $\mu'$.  

\begin{example}[Multiplication with finite rank polynomial]
Let $\mfq=(q_d)_{d=s}^\infty$ be as in Example \ref{exm:deviation_from_dirac}. 
Then, from Example \ref{exm:lim.dist.finite.pertubation}, we know that 
$\mfq=(q_d)_{d=s}^\infty$ has infinitesimal distribution  $(\nu,\nu')=\left( \delta_0,-t \delta_0 + \sum_{k=1}^{t} \delta_{\beta_k}\right)$. 
In this case, \eqref{eqn:rhats_multiplicative} from Theorem \ref{thm:infinitesimal_multiplicative_finite_free} reduces to 
\begin{equation*}
\widehat r_n(\mfp \boxtimes \mfq)
	=
\widehat r_n ( \mfq ) \  (r_1(\mu))^{n} \, .
\end{equation*}
Indeed, since $r_n(\delta_0)=0$ for every integer $n \geq 1$, 
we have that $r_{Kr(\pi)}(\nu)=0$ for any partition $\pi \in NC(n)$ and that $r_{Kr_{t,s}(\sigma)}(\nu)=0$ for any permutation $\sigma \in S_{NC}(t,s)$ with $t+s=n$; meanwhile, $\widehat{r}_{|W|}( \mfq ) \, r_{Kr(\pi)\backslash W}(\nu) \cdot r_\pi(\mu)  = 0$ unless $\pi=0_n$ and $\abs{W} = n$, or equivalently, $Kr(\pi) =1_n$ and $W=[n]$.  
Consequently, $\widehat r_n ( \mfq ) \  (r_1(\mu))^{n}$ amounts to the only possibly nonzero term in the right-hand side of \eqref{eqn:rhats_multiplicative}. 
It then follows that 
\begin{equation*}
	\widehat{R}_{\mfp \boxtimes \mfq} (z) 
=  
		\sum_{n=1}^{\infty} \widehat{r}_{n}(\mfp \boxtimes \mfq) \, z^{n-1} 
= 
	r_1(\mu) \cdot \widehat{R}_{\mfq} (r_1(\mu) \, z ) \ .
\end{equation*}
Now, from Theorem \ref{prop:cauchy_multiplicative_finite_free}, we have that $\mfp \boxtimes \mfq = (p_d\boxtimes_d q_d)_{d\geq1}$ has infinitesimal distribution  $(\rho,\rho')$  where $\rho=\mu\boxtimes\delta_0$ and $\rho'$ has Cauchy transform $G_{\rho'}$ given by 
\begin{equation*} 
	G_{\rho'}  
=  
-  (\widehat{R}_{\mfp \boxtimes \mfq} \circ G_{\mu \boxtimes \delta_0} ) \cdot G'_{\mu \boxtimes \delta_0} 
-
\gdos_{\mu \boxtimes \delta_0} \, .
\end{equation*}
Moreover, from Example \ref{exm:lim.dist.finite.pertubation}, we also know that $G_{\mu \boxtimes \delta_0}=1/z$, $G'_{\mu \boxtimes \delta_0}=-1/{z^{2}}$, and $\gdos_{\mu \boxtimes \delta_0} = 0$ since $\mu \boxtimes \delta_0 = \delta_0$. 
Therefore, since $r_1(\mu) = m_1(\mu)$, we obtain that 
\begin{equation*}
	G_{\rho'}  
=  
	m_1(\mu) \cdot \widehat{R}_{\mfq} \left( \, \frac{\, m_1(\mu) \, }{z} \, \right) \cdot \frac{1}{z^2}
=
	m_1(\mu) \cdot 
	\sum_{k=1}^{t}
 \sum_{n=1}^{\infty} 
   \frac{(\beta_k \, m_1(\mu))^{n}  }{z^{n+1}} \, ,
\end{equation*}
or equivalently, $\rho' =  -t m_1(\mu) \delta_0 +  m_1(\mu) \sum_{k=1}^{t} \delta_{ m_1(\mu) \beta_k} $. 
\end{example}

The Dirac delta distribution $\delta_1$ is the identity for the free multiplicative convolution, and it is also the limiting root distribution of  $\mfp=(p_d)_{d\geq1}$ and $\mfq=(q_d)_{d\geq1}$ when they both have all but finitely many roots located at $1$. 
This is the analogue to Example \ref{exm:deviation_from_dirac2} in the multiplicative case. 
For this, we obtain not only that $\widehat{R}_{\mfp \boxtimes \mfq} = \widehat{R}_{\mfp}+ \widehat{R}_{\mfq}$, but also that the infinitesimal moments of the convolution equal the sum of the individual infinitesimal moments, just as in Example \ref{exm:deviation_from_dirac2}.

\begin{example}[Multiplication of finite rank perturbation of identity] 
\label{exm:two_perturbations_identity}

Assume that each $p_d$ and each $q_d$ from the sequences $\mfp=(p_d)_{d=s}^\infty$ and $\mfq=(q_d)_{d=t}^\infty$ are given by 
\begin{equation*}
	p_d(x)=(x-1)^{d-s}(x-\alpha_1)(x-\alpha_2) \cdots (x-\alpha_s) 
\qquad \text{and}\qquad 
	q_d(x)=(x-1)^{d-t}(x-\beta_1) \cdots (x-\beta_t)
\end{equation*}
for fixed integers $s,t \geq 1$ and fixed real numbers  $\alpha_1,\dots, \alpha_s, \beta_1,\dots,\beta_t\in \mathbb{R}$. 
From Example \ref{exm:lim.dist.finite.pertubation}, we know that $\mfp=(p_d)_{d=s}^\infty$ and $\mfq=(q_d)_{d=t}^\infty$ have infinitesimal distributions respectively given by
\begin{equation*}
\textstyle
(\mu,\mu') =
\left( \delta_1,-s \delta_1 + \sum_{k=1}^{s} \delta_{\alpha_k}\right)
\quad \text{and} \quad
(\nu,\nu') = 
\left(\delta_1,-t \delta_1 + \sum_{k=1}^{t} \delta_{\beta_k}\right).
\end{equation*}
%
%
This implies that $m_n(\mu)=m_n(\nu)=1$ for all integers $n \geq 1$; and hence, we obtain that  
$$r_n(\mu)=r_n(\nu)= \begin{cases} 1, & \text{if } n=1, \\ 0, & \text{if } n\geq 2.\end{cases}$$
Now, since $Kr:NC(n) \to NC(n)$ is a bijection, \eqref{eqn:rhats_multiplicative} from Theorem \ref{thm:infinitesimal_multiplicative_finite_free} can be rewritten as 
\begin{align}
\widehat r_n( \mfp \boxtimes \mfq )
    &= 
        \sum_{\pi\in NC(n)}
            \sum_{V\in\pi} 
		          \Big(
                        \widehat{r}_{|V|} (\mfp) 
                        r_{\pi \setminus V}(\mu)
                        \cdot r_{Kr(\pi)}(\nu) 
        +      
                        \widehat{r}_{|V|} (\mfq) 
                        r_{{Kr}({Kr}(\pi)) \setminus V}(\nu)
                        \cdot r_{{Kr}(\pi)}(\mu) \Big)
       \label{eqn:rhats_multiplicative2}   \\ 
    & \quad  -
            \frac{n}{2}
            \sum_{\substack{ t+s = n \\ \sigma \in S_{NC}(t,s) } } 
                    \frac{r_{\sigma}(\mu) \, r_{Kr_{t,s}(\sigma)}(\nu)}{ts} \nonumber
\end{align}
However, since every $ \sigma \in S_{NC}(t,s)$ has at least one block with two or more elements, which implies $r_{\sigma} (\mu) = 0$, no permutation from $ S_{NC}(t,s)$ contributes to right-hand side of \eqref{eqn:rhats_multiplicative2}.
Additionally, given a partition $\pi \in NC(n)$, if its Kreweras complement $Kr(\pi)$ contains at least one block with two or more elements, then $r_{Kr(\pi)} (\mu) = r_{Kr(\pi)} (\nu) = 0$. 
Thus, a partition $\pi \in NC(n)$ contributes to right-hand side of  \eqref{eqn:rhats_multiplicative2} only if all the blocks of $Kr(\pi)$ are singletons, which is equivalent to $\pi = 1_n$. 
Therefore, we obtain that 
\begin{equation*}
    \widehat{r}_n( \mfp \boxtimes \mfq ) 
    = 
    \widehat{r}_n( \mfp ) + \widehat{r}_n(\mfq ) , 
\text{\qquad or equivalently, \qquad}
\widehat{R}_{\mfp \boxtimes \mfq} = \widehat{R}_{\mfp}+ \widehat{R}_{\mfq} \, .
\end{equation*}
It follows from Theorem \ref{prop:cauchy_multiplicative_finite_free} that the sequence $ \mfp \boxtimes \mfq = (p_d\boxtimes_d q_d)_{d\geq1}$ has infinitesimal distribution $(\rho,\rho')$ where $\rho=\delta_1 \boxtimes \delta_1 = \delta_1$ and $\rho'$ has Cauchy transform $G_{\rho'}$ given by 
\begin{equation*}
G_{\rho'} =  
-  (\widehat{R}^{}_{\mfp} \circ G^{}_{\delta_1} ) \cdot G'_{\delta_1} 
-  (\widehat{R}^{}_{\mfq} \circ G^{}_{\delta_1} ) \cdot G'_{\delta_1}
+ \gdos^{}_{\delta_1} . 
\end{equation*}
But, from Example \ref{exm:lim.dist.finite.pertubation}, we know that $-  (\widehat{R}^{}_{\mfp} \circ G^{}_{\delta_1} ) \cdot G'_{\delta_1} $ and $
-  (\widehat{R}^{}_{\mfq} \circ^{} G_{\delta_1} ) \cdot G'_{\delta_1}$ are the Cauchy transforms of $-s \delta_1 + \sum_{k=1}^{s} \delta_{\alpha_k}$ and $-t \delta_1 + \sum_{k=1}^{t} \delta_{\beta_k}$, respectively, and that $H^{}_{\delta_1}=0$. 
Consequently, we obtain that 
\begin{equation*}
G_{\rho'}(z) =
     - \frac{s+t}{z-1} + \sum_{k=1}^{s} \frac{1}{z-\alpha_k} + \sum_{k=1}^{t} \frac{1}{z-\beta_k}  ,
\end{equation*}
or equivalently, $\rho' = - \, (s+t) \, \delta_1  + \sum_{k=1}^{s} \delta_{\alpha_k} + \sum_{k=1}^{t} \delta_{\beta_k} $. 
\end{example}

The previous example relied on Theorem \ref{thm:infinitesimal_multiplicative_finite_free}. 
A more general case, namely, when we only assume that $\mfq=(q_d)_{d\geq1}$ has all but finitely many roots located at $1$, can be addressed by considering Theorem \ref{prop:moments_for_multiplicative} instead.  
As in the additive case, see Example \ref{exm:deviation_from_dirac}, we show that our combinatorial results also retrieve, as particular cases, the distributions examined by Shlyakhtenko in \cite[Section 4.2]{shlyakhtenko2018free}, corresponding this time to the infinitesimal free multiplicative convolution. 
An alternative approach to obtain the following is Corollary \ref{Cor. infmult}.

\begin{example}[Multiplication by a finite rank perturbation]
\label{exm:deviation_from_dirac3}
Let $\mfq=(q_d)_{d=s}^\infty$ be as in Example \ref{exm:two_perturbations_identity}. We know that  $\mfq=(q_d)_{d=s}^\infty$  has infinitesimal distribution  $(\nu,\nu')=\left( \delta_1,-t \delta_1 + \sum_{k=1}^{t} \delta_{\beta_k}\right)$. 
Since $Kr:NC(n) \to NC(n)$ is a bijection, \eqref{eqn:mprime_multiplicative} from Theorem \ref{prop:moments_for_multiplicative} can be rewritten as
\begin{align*}
m'_n(\mfp \boxtimes \mfq)
& = 
		\sum_{\pi\in NC(n)} 
		\sum_{V\in\pi}
			\Big(
					 m'_{|V|}(p) \, m_{\pi \backslash V}(\mu) \cdot r_{Kr(\pi)}(\nu)
				+
					\widehat{r}_{|V|}(q) \, r_{ \pi \backslash V}(\nu) \cdot m_{{Kr}^{-1}(\pi)}(\mu)
			\Big)  \\
&\quad 
-\frac{n}{2} \sum_{\substack{ t+s = n \\ \sigma \in S_{NC}(t,s) } } 
\frac{m_{\sigma}(\mu) \, r_{Kr_{t,s}(\sigma)}(\nu)}{ts}
\end{align*}
As in Example \ref{exm:two_perturbations_identity}, since $r_n(\nu)=0$ for $n \geq 2$, we have that $r_{Kr(\pi)}(\nu)=0$ unless $\pi = 1_n$ and that $r_{Kr_{t,s}(\sigma)}=0$ for every permutation $ \sigma \in S_{NC}(t,s)$; 
additionally, $r_{\pi \setminus V}(\nu)=0$ unless $\pi\setminus V$ contains only singletons as blocks. 
Consequently, the previous equation yields 
\begin{equation*}
m'_n( \mfp \boxtimes \mfq)
   \ \ \ =    \ \ \
        m'_n(\mfp)
       \ \ \ +
        \sum_{ \emptyset \neq S \subseteq [n] }
        	\widehat{r}_{\abs{S}}(\mfq)  \, 
             r_{0_{n\setminus S}}(\nu)
			\cdot 
            m_{{Kr}^{-1}(S \cup 0_{n\setminus S})}(\mu)
\end{equation*}
where we have used the fact that $S \mapsto (0_{n\setminus S} \cup S,S)$ is a bijection from nonempty sets $S \subseteq [n]$ to pairs $(\pi,V)$ where $V \in \pi \in NC(n)$ and $ \pi\setminus V$ contains only singletons. 
The map  $S \mapsto {Kr}^{-1}(S \cup 0_{n\setminus S})$ is a bijection from nonempty sets $S \subseteq [n]$ to the set of cyclic interval partitions $CI(n)$ such that $\abs{S}= \blocks{\pi}$ provided $\pi = {Kr}^{-1}(S \cup 0_{n\setminus S})$. 
Hence, we obtain that 
\begin{equation}\label{eqn:moments_finiterank_pert_2}
m'_n( \mfp \boxtimes \mfq)
   \ \ \  =    \ \ \
        m'_n(\mfp)
       \ \ \ 
+
        \sum_{\pi \in CI(n)}
        \widehat{r}_{\blocks{\pi}}(\mfq) \cdot m_{\pi} (\mu) 
       \ \ \ =    \ \ \
        m'_n(\mfp)
       \ \ \ 
+		
		\sum_{k=1}^{t}
        \sum_{\pi \in CI(n)}
        (\beta_k-1)^{\blocks{\pi}} \cdot m_{\pi} (\mu) \, .
\end{equation}
where the last equality follows from the fact that $\widehat{r}_{n}(\mfq) =	\sum_{k=1}^{t} (\beta_k-1)^{n}$, see Example \ref{exm:lim.dist.finite.pertubation}. 

We will now compute the Cauchy transform associated to each of the rightmost terms in \eqref{eqn:moments_finiterank_pert_2}. 
To this end, for a fixed $k \in [t]$, take  
\begin{equation*}
\tilde{m}_n = -\, (\beta_k-1) \, m_n(\mu)
\text{\qquad and \qquad}
\overline{m}_n = 
		\sum_{\pi \in CI(n)}
		 - \, (\beta_k-1)^{\blocks{\pi}} \cdot m_{\pi} (\mu)
\text{\qquad for every \qquad} n \geq 1
\end{equation*}
and consider 
\begin{equation*}
\widetilde{G}(z) = \frac{1}{z} + \sum_{n=1}^{\infty} \frac{\tilde{m}_n}{z^{n+1}} = (1-\beta_k)G_{\mu} (z) + \frac{\beta_k}{z}
\text{\qquad and \qquad}
\overline{G}(z) = \frac{1}{z} + \sum_{n=1}^{\infty} \frac{\overline{m}_n}{z^{n+1}}
	=
		\frac{1}{z}  + \overline{g}(z). 
\end{equation*}
Note that the coefficients $m_n(\mu)$, $\tilde{m}_n$, and $\overline{m}_n$ satisfy the relation   
\begin{equation*}
\sum_{\pi \in CI(n)} (-1)^{\blocks{\pi}+1} \tilde{m}_{\pi} 
	= 
  \sum_{\pi \in CI(n)} -(\beta_k-1)^{\blocks{\pi}} m_{\pi}(\mu)
=
\overline{m}_n \, .
\end{equation*}
This means that $\widetilde{G}(z)$ is in fact the Markov-Krein transform of $\overline{G}(z)$, so they are related through the equation  
\begin{equation*}
-\overline{G}(z) = \frac{\widetilde{G}'}{\widetilde{G}} = -\frac{1}{z} - \overline{g}(z) 
\, .
\end{equation*}
Thus, we obtain that 
\begin{equation*}
-\overline{g}(z) 
 = 
\frac{\widetilde{G}'}{\widetilde{G}} + \frac{1}{z} 
= 
\frac{d}{dz} \log( \, z \widetilde{G} \, )
		- \sum_{n=1}^{\infty} \frac{\overline{m}_n}{z^{n+1}} 
=
    \frac{d}{dz} \log( \, (1-\beta_k)zG_{\mu}(z) + \beta_k \, )
\end{equation*}
Therefore, the sequence $ \mfp \boxtimes \mfq = (p_d\boxtimes_d q_d)_{d\geq1}$ has infinitesimal distribution $(\rho,\rho')$ where $\rho=\mu \boxtimes \delta_1 = \mu$ and $\rho'$ has Cauchy transform $G_{\rho'}$ given by 
\begin{equation*}
G_{\rho'} 
=  
G_{\mu'} + \sum_{k=1}^{t} 	\frac{d}{dz}  \log \left( z(1-\beta_k)G_{\mu} + \beta_k \right) .
\end{equation*}
In particular, if $\beta_k=0$ for every $k \in [t]$, the previous formula reduces to 
\begin{equation*}
G_{\rho'} =  G_{\mu'} + t \frac{d}{dz} \log(zG_{\mu}(z))
=  G_{\mu'} + t \left( \frac{G'_{\mu}}{G_{\mu}} + \frac{1}{z} \right).
\end{equation*}
\end{example}

Our next example is motivated by the work \cite{fujiehasebe2022} where Fujie and Hasebe calculate the fluctuations of the eigenvalue distribution of the principal minor of  unitarily invariant matrices with limiting eigenvalue distribution. 

\begin{example} \label{ex.derivatives} 
Assume that $\mu \neq  \delta_0$. 
Fix an integer $s \geq 0$ and let $p_d^{(s)}$ denote the $s$-th  derivative of $p_d$ for $s \leq d$. 
We are interested in analyzing   
\begin{equation*}
\tau_d : = d\mu-(d-s)\meas{p_d^{(s)}}=d\mu- \sum_{k=1}^{d-s} \delta_{\lambda_k(p_d^{(s)})}
\quad 
\text{as}
\quad
d \to \infty. 
\end{equation*}
From \cite[Lemma 3.5]{arizmendi2023finite}, we know that $p_d^{(s)}$ can be computed via finite-free multiplicative convolution. 
Concretely, letting $q_d(x) =x^s (x-1)^{d-s}  $ for $s \leq d$, we have that 
\begin{equation*}
p_d\boxtimes_d q_d =\frac{1}{(d)_s}x^sp_d^{(s)}(x) \, .
\end{equation*}
Note that $\mfq=(q_d)_{d=s}^\infty$ has infinitesimal distribution  $(\nu,\nu')=\left( \delta_1, s \delta_{0} -s \delta_1 \right)$ in this case. 
Thus, from Proposition \ref{prop:finite_free_infinitesimal_coincide_mult} and Corollary \ref{Cor. infmult}, we obtain that $\mfp \boxtimes\mfq = (p_d\boxtimes_d q_d)_{d\geq1}$ has infinitesimal distribution $(\gamma,\gamma')$ where $\gamma = \mu$ and 
\begin{equation*}
 G_{\gamma'} 
    =
        G_{\mu'}
    +
        G_{\nu'} (\theta_{\mu}) \cdot \theta'_{\mu}
\quad \text{with}\quad 
\theta_{\mu}(z) =\frac{G_{\mu}(z)}{ G_{\mu}(z) - 1/z}
\quad \text{and}\quad 
\theta'_{\mu} (z) = -\frac{z G_\mu'(z) + G_\mu(z)}{(z G_\mu(z) - 1)^2} \, .
\end{equation*}
Moreover, since $G_{\nu'}(z)=\frac{k}{z}-\frac{k}{z-1}$, we also obtain that 
\begin{equation*}
 G_{\nu'}(\theta_{\mu}(z))=-\frac{s \, (z G_\mu(z) - 1)^2}{z G_\mu(z)} \, .
\end{equation*}
Therefore, we get that
\begin{equation*}
 G_{\gamma'} 
    =
        G_{\mu'}(z)+s\frac{z G_\mu'(z) + G_\mu(z)}{z G_\mu(z)}
=
		G_{\mu'}(z)+\frac{s G_\mu'(z)}{G_\mu(z)} + \frac{s}{z}.
\end{equation*}
Finally, to obtain the limit of $\tau_d$ as $d \to \infty$, that we denote by $\tau$, we note that
\begin{equation*}
dG_{\meas{p_d\boxtimes_d q_d}}-(d-s) G_{\meas{p_d^{(s)}}}=\frac{s}{z} 
\end{equation*}
which implies 
\begin{equation*}
G_{\tau}=G_{\gamma'} -\frac{d}{z}=G_{\mu'}(z)+\frac{sG_\mu'(z)}{G_\mu(z)}.
\end{equation*}%
\end{example}

Our last example relates to the inverse of Laguerre polynomials with respect to the finite-free multiplicative convolution. 

\begin{example}[Multiplicative inverse of Laguerre polynomials]
\label{exm:inverse_laguerre}
In this example, we will look into their inverses with respect to the multiplicative convolution. Namely, we will analyze the polynomials 
\begin{equation*}
p_d(x)
= 
\sum_{i=0}^d (-1)^i x^{d-i} \binom{d}{i} \frac{d^i}{(d)_{i}}
, 
\end{equation*}
which can be equivalently defined through the relation $p_d \boxtimes_d  L^{(1)}_d = (x-1)^d $. 
From \cite[Lemma 3.2]{arizmendi2023finite}, we know that the moments of $p_d$ coincide with the cumulants of $p_d \boxtimes_d  L^{(1)}_d$. Thus, we obtain that 
\begin{equation*}
m_n(p_d) = \kappa^d_n( (x-1)^d ) 
=
	\begin{cases}
			1,	& \text{if }n=1, \\
			0,	& \text{otherwise}.
			\end{cases}  
\end{equation*}
Although, the sequence $(m_n(p_d))_{n=1}$ is independent from $d$, it does not correspond to the moments of a measure supported on $\rr$. 
Indeed, there is no $\mu \in \mathcal{P}(\rr)$  such that $m_1(\mu)=1$ and $m_2(\mu)=0$. 
Nonetheless, one can still consider the ``Cauchy transform'' as the formal power series that corresponds to this sequence: 
\begin{equation*}
G_{\mfp}(z)
:=
		\frac{1}{z} + \sum_{n=1}^{\infty} \frac{m_n(p_d)}{z^{n+1}}
= \frac{1}{z}+\frac{1}{z^2}
= 
		\frac{z+1}{z^2} \, .
\end{equation*}
The first and second derivatives of $G_{\mfp}(z)$ are then given by 
\begin{equation*}
G_{\mfp}'(z)=-\frac{1}{z^2}-\frac{2}{z^3}= \frac{-z-2}{z^3}
\qquad \text{and} \qquad 
G_{\mfp}''(z)=\frac{2}{z^3}+\frac{6}{z^4}= 2 \frac{z+3}{z^4} \, .
\end{equation*}
And $\gdos$-transform takes the particular nice expression 
\begin{equation*}
\gdos_{\mfp}(z) 
	= 
	\frac{G_{\mfp}'(z)}{G_{\mfp}(z)} -\frac{G_{\mfp}''(z)}{2G_{\mfp}'(z)}
=
		-\frac{z+2}{z(z+1)}+\frac{z+3}{z(z+2)}
=	
	\frac{-1}{z(z+1)(z+2)}
\end{equation*}
Now, since the Cauchy transform $G_{\mfp}(z)$ satisfies $z^2G_{\mfp}(z)-z-1 =0$, its inverse $K_{\mfp}(z) $ under composition is given by
\begin{equation*}
K_{\mfp}(z)=\frac{1+\sqrt{ 1+4 z}}{2z} 
\text{\quad with \quad } 
K_{\mfp}'(z)
=
		\frac{1}{z\sqrt{ 1+4 z}}-\frac{1+\sqrt{ 1+4 z}}{2z^2}
=	
		-\frac{1 + 2 z + \sqrt{1 + 4 z}}{2 z^2 \sqrt{1 + 4 z}} \, .
\end{equation*}
Thus, from Example \ref{exm:lim.dist.fixed.moments}, we obtain that 
\begin{equation*}
\widehat{R}_\mfp(z) 
	= 	
		-\gdos_{\mfp}(K_{\mfp}(z) ) \cdot K_{\mfp}'(z)
	=
		\frac{-4z}{(1+ \sqrt{1 + 4 z})(1 + 4z + \sqrt{1 + 4 z})\sqrt{1 + 4 z}}
	=
		\frac{-4 z}{(1 + 4z + \sqrt{1 + 4 z})^2} \, 
\end{equation*}
with Taylor series expansion 
\begin{equation*}
\widehat{R}_\mfp(z)
= 
		\sum_{n=1}^{\infty} \hat{r}_n(\mfp) \, z^{n-1}
=
		0-1z + 6 z^2 - 29 z^3 + 130 z^4 - 562 z^5 + 2380 z^6+ \dots 
	 \ . 
\end{equation*}
where $\hat{r}_n(\mfp)=(4^{n} - \binom{2n+1}{n})(-1)^n$. 
Let us point out that these cumulant fluctuations coincide up to a sign with the sequence A008549 in Sloane's OEIS. 
They are also related to the enumeration of certain annular partitions arising in the infinitesimal distribution of Real Wishart Random Matrices, see \cite[Corollary 36]{mingo2025asymptotic}.
\end{example}

\section*{Acknowledgements}

Some of the discussions of this project take place at the evet SIMA 2024, held at Merida, Mexico. Support for this research was provided by an AMS-Simons Travel Grant, to support the trip of D.P. to SIMA, and a visit of O.A. to Texas A\&M University. 
O.A. gratefully acknowledges financial support by the grant Conacyt A1-S-9764. 
D.P. appreciates the hospitality of CIMAT during November 2023.


\end{document}